\documentclass{amsart}

\usepackage{amssymb}

\usepackage{MnSymbol}

\usepackage{graphicx}               

\usepackage{bbm}

\usepackage{mathtools}

\usepackage{hyperref}

\usepackage{extarrows}

\usepackage{color}

\newtheorem{theorem}{Theorem}[section]
\newtheorem{lemma}[theorem]{Lemma}
\newtheorem{proposition}[theorem]{Proposition}
\newtheorem{corollary}[theorem]{Corollary}

\theoremstyle{definition}
\newtheorem{definition}[theorem]{Definition}

\theoremstyle{remark}
\newtheorem{remark}[theorem]{Remark}

\numberwithin{equation}{section}

\newcommand{\A}{\mathbb{A}}
\newcommand{\C}{\mathbb{C}}
\newcommand{\Q}{\mathbb{Q}}
\newcommand{\R}{\mathbb{R}}
\newcommand{\Z}{\mathbb{Z}}

\newcommand{\roi}{\mathcal{O}}

\newcommand{\gn}{\mathfrak{n}}

\newcommand{\gp}{\mathfrak{p}}
\newcommand{\gf}{\mathfrak{f}}
\newcommand{\gq}{\mathfrak{q}}
\newcommand{\gm}{\mathfrak{m}}
\newcommand{\gM}{\mathfrak{M}}

\newcommand{\matrixxx}[2]{\left(\begin{matrix}
#1\\
#2
\end{matrix} \right)}

\newcommand{\smallmatrixx}[4]{(\begin{smallmatrix}
#1 & #2\\
#3 & #4
\end{smallmatrix})}
\space

\newcommand{\matrixx}[4]{\left(\begin{matrix}
#1 & #2\\
#3 & #4
\end{matrix} \right)}
\space

\newcommand{\tensorspace}{\roi_{K,p}}

\begin{document}

\title{Non-cuspidal Bianchi modular forms and Katz $p$-adic $L$-functions}


\author{Luis Santiago Eduardo Palacios Moyano}
\address{}
\curraddr{}
\email{}
\thanks{}



\date{}

\dedicatory{}

\begin{abstract}
Let $K$ be an imaginary quadratic field. In this article, we construct $p$-adic $L$-functions of non-cuspidal Bianchi modular forms by introducing the notions of $C$-cuspidality and partial Bianchi modular symbols. 

When $p$ splits in $K$, we focus on $p$-adic $L$-functions of non-cuspidal base change Bianchi modular forms, showing that they factor as products of two Katz $p$-adic $L$-functions. 
\end{abstract}

\maketitle

\section{Introduction}

Let $p$ be a fixed prime number. The study of $p$-adic $L$-functions has become a cornerstone of modern number theory, as these functions often serve as a vital link between algebraic number theory, arithmetic geometry, and analysis.

In \cite{pollack2011}, Pollack and Stevens constructed the $p$-adic $L$-function of cuspidal modular forms using the theory of overconvergent modular symbols. Their work has since inspired numerous generalizations (see for instance \cite{barrera2018}, \cite{chris2017}, \cite{canadian}, for the case of $\mathrm{GL}_2$). For non-cuspidal modular forms, Bellaïche and Dasgupta introduced the notions of $C$-cuspidality and partial modular symbols in \cite{bellaiche2015p}, which allowed them to construct the $p$-adic $L$-function of Eisenstein series.

Now, let $K$ be an imaginary quadratic field, and consider automorphic forms for $\mathrm{GL}_2$ over $K$, commonly referred to as Bianchi modular forms. For the cuspidal case, Williams constructed the $p$-adic $L$-function in \cite{chris2017} by developing the theory of Bianchi modular symbols.

In this article, we extend these ideas to the non-cuspidal setting, constructing the $p$-adic $L$-function for non-cuspidal Bianchi modular forms by combining and adapting the approaches of Williams and Bellaïche-Dasgupta. Additionally, we establish a significant connection between Katz's $p$-adic $L$-functions and our construction in the context of non-cuspidal base change Bianchi modular forms.

\subsection{\texorpdfstring{$C$}{Lg}-cuspidal Bianchi modular forms and partial symbols}{\label{s1 intro}}

For our study of non-cuspidal Bianchi modular forms, in Section \ref{Fourier} we generalize to the Bianchi setting the notion of \textit{$C$-cuspidality} given in \cite{bellaiche2015p}. Such property is related with the vanishing of constant terms of Fourier expansions at suitable cusps and for Bianchi modular forms with level at $p$ we have: 
\begin{equation*}
\{\text{cuspidal}\}\subset\{C\textrm{-cuspidal}\}\subset\{\text{Bianchi modular forms}\}. 
\end{equation*} 

For cuspidal Bianchi modular forms, an integral formula is known for their complex $L$-functions. In Proposition \ref{T4.6}, we establish a similar formula for $C$-cuspidal Bianchi modular forms.

In Section \ref{S3}, we introduce algebraic analogues of $C$-cuspidal Bianchi modular forms, called \textit{partial Bianchi modular symbols}, that are easier to study $p$-adically. To define them, we generalize the classical partial modular symbols introduced in \cite{bellaiche2015p} and adapt the Bianchi modular symbols of parallel weight used in \cite{chris2017}. In Proposition \ref{attach partial}, we prove that we can attach, in a Hecke-equivariant way, a partial Bianchi modular symbol to a $C$-cuspidal Bianchi modular form.

To link partial Bianchi modular symbols with spaces of $p$-adic distributions, we introduce the \textit{overconvergent partial Bianchi modular symbols} in Section \ref{section overconvergent} and generalizing \cite{chris2017}, we obtain a classicality result in Proposition \ref{partialcontrol}. 

In Section \ref{S4.4}, we construct the $p$-adic $L$-function of a $C$-cuspidal Bianchi eigenform $\mathcal{F}$ of weight $(k,\ell)$ and level $K_0(\gn)$, with $U_\mathfrak{p}$-eigenvalues $\lambda_\mathfrak{p}$ having suitable $p$-adic valuation (see Definition \ref{D10.4}). For this, we first attach to $\mathcal{F}$ a complex-valued partial Bianchi modular eigensymbol $\phi_\mathcal{F}$ using Proposition \ref{attach partial}. Then, using an isomorphism $\iota:\mathbb{C}\xrightarrow{\sim}\overline{\mathbb{Q}}_p$ (satisfying $\iota\circ\iota_\infty=\iota_p$ with some previously fixed embeddings $\iota_\infty:\overline{\Q}\hookrightarrow\C, \iota_p:\overline{\Q}\hookrightarrow\overline{\Q}_p$), we view $p$-adically the values of $\phi_\mathcal{F}$ and then lift it uniquely to an overconvergent partial Bianchi modular eigensymbol $\Psi_\mathcal{F}$ using Proposition \ref{partialcontrol}. After taking the Mellin transform of $\Psi_\mathcal{F}$, we obtain (see Theorem \ref{C-cuspidal p-adic} for the precise statement): 
\begin{theorem}\label{C-cuspidal p-adic intro}
There exists a locally analytic distribution $L_p^{\iota}(\mathcal{F},-)$ on the ray class group $\mathrm{Cl}_K(p^{\infty})$, such that for any Hecke character $\psi$ of $K$ of conductor $\mathfrak{f}|(p^\infty)$ and infinity type $(q,r)$ satisfying $0\leqslant q \leqslant k, 0\leqslant r \leqslant \ell$, we have 
\begin{equation*}
L_p^{\iota}(\mathcal{F},\psi_{p-\mathrm{fin}})=\left[\prod_{\mathfrak{p}|p}\left(1-\frac{1}{\lambda_\gp\psi(\gp)}\right)\right]\left[ \frac{D_Kw_K\tau(\psi)}{(-1)^{\ell+q+r}2\lambda_\mathfrak{f}} \right] \Lambda(\mathcal{F},\psi),
\end{equation*}
where $\psi_{p-\mathrm{fin}}$ is the $p$-adic avatar of $\psi$, $-D_K$ is the discriminant of $K$, $w_K=|\roi_K^\times|$, $\tau(\cdot)$ is a Gauss sum, $\lambda_\gf$ is the eigenvalue of $U_\gf=\prod_{\gp|p}U_\gp^{r_\gp}$ and  $\Lambda(\mathcal{F},-)$ is the normalized $L$-function of $\mathcal{F}$ as in (\ref{gammafactors}).

The distribution $L_p^{\iota}(\mathcal{F},-)$ satisfies suitable growth conditions (see Section \ref{admissible}) and therefore is unique.
\end{theorem}

By part ii) of Remark \ref{turn-C-cuspi}, we can construct the $p$-adic $L$-function of certain non-cuspidal Bianchi modular forms by turning them into $C$-cuspidal forms and applying the previous theorem.

\subsection{Non-cuspidal base change and Katz \texorpdfstring{$p$}{Lg}-adic \texorpdfstring{$L$}{Lg}-functions}

A case of historical interest where we can apply the methods outlined in Section \ref{s1 intro} to construct $p$-adic $L$-functions is the non-cuspidal base change scenario.

Suppose that $p$ splits in $K$ and let $\varphi$ be a Hecke character of $K$ with conductor coprime to $p$ and infinity type $(-k-1,0)$ with $k\geqslant0$. Denote by $\theta_{\varphi}$ the theta series attached to $\varphi$, which is a cuspidal modular form of weight $k+2$ and let $\theta_{\varphi/K}$ be the base change to $K$ of $\theta_{\varphi}$, which is known to be a non-cuspidal Bianchi modular form of weight $(k,k)$. In Proposition \ref{basechangequasi}, we prove that the ordinary $p$-stabilization of $\theta_{\varphi/K}$, which we denote by $\theta_{\varphi/K}^p$, is a $C$-cuspidal Bianchi modular form. In particular, by Theorem \ref{C-cuspidal p-adic intro}, we can obtain its $p$-adic $L$-function $L_p^{\iota}(\theta_{\varphi/K}^p,-)$. 

Part of our interest in the Bianchi modular form $\theta_{\varphi/K}^p$ relies on the fact that we can avoid the use of an isomorphism $\iota:\mathbb{C}\xrightarrow{\sim}\overline{\mathbb{Q}}_p$ to construct its $p$-adic $L$-function, and additionally we can factorize its $p$-adic $L$-function as the product of two Katz $p$-adic $L$-functions.

In Section \ref{basechangeLfunction}, we study the $L$-function of $\theta_{\varphi/K}$ and prove that it factors as the product of two Hecke $L$-functions in Lemma \ref{Prop 9.1}. Using such factorization, combined with an algebraicity result of Hecke $L$-functions, we show the existence of a \textit{complex period} $\Omega_{\theta_{\varphi/K}}$, which allows us to prove algebraicity of critical $L$-values of $\theta_{\varphi/K}$ and $\theta_{\varphi/K}^p$ in Proposition \ref{period base change} and Corollary \ref{lemma6.9}, respectively.

Let $\phi_{\theta_{\varphi/K}^p}$ be the complex-valued partial Bianchi modular symbol attached to $\theta_{\varphi/K}^p$ in Proposition \ref{attach partial}, then by defining $\phi_{\theta_{\varphi/K}^p}^\mathrm{alg}:=\phi_{\theta_{\varphi/K}^p}/\Omega_{\theta_{\varphi/K}}$ we prove in Proposition \ref{Prop6.7} that $\phi_{\theta_{\varphi/K}^p}^\mathrm{alg}$ has algebraic values and consequently, we can view it as having $p$-adic values. Finally, using Proposition \ref{partialcontrol}, we can lift $\phi_{\theta_{\varphi/K}^p}^\mathrm{alg}$ to an overconvergent Bianchi modular symbol and taking its Mellin transform we obtain the following (see Theorem \ref{Cor8.3}):

\begin{theorem}{\label{C1.2}} 
There exists a unique measure $L_p(\theta_{\varphi/K}^p,-)$ on the ray class group $\mathrm{Cl}_K(p^{\infty})$ such that for any Hecke character $\psi$ of $K$ of conductor $\gf=\gp^t\overline{\gp}^s$ and infinity type $(q,r)$ satisfying $0\leqslant q,r \leqslant k$, we have 
\begin{equation*}
L_p(\theta_{\varphi/K}^p,\psi_{p-\mathrm{fin}})=\left[\prod_{\gq|p}\left(1-\frac{1}{\varphi(\overline{\gp})\psi(\gq)}\right)\right]\left[ \frac{D_Kw_K\tau(\psi)}{(-1)^{k+q+r}2 \varphi(\overline{\gp})^{t+s}\Omega_{\theta_{\varphi/K}}} \right] \Lambda(\theta_{\varphi/K}^p,\psi).
\end{equation*}
\end{theorem} 

The $p$-adic $L$-functions $L_p(\theta_{\varphi/K}^p,-)$ and $L_p^\iota(\theta_{\varphi/K}^p,-)$, can be related, more precisely, in Proposition \ref{equality of measures} we prove the following equality of measures: 
\begin{equation*}
L_p(\theta_{\varphi/K}^p,-)=\frac{L_p^\iota(\theta_{\varphi/K}^p,-)}{\iota(\Omega_{\theta_{\varphi/K}})}.
\end{equation*}

The factorization of the $L$-function of $\theta_{\varphi/K}$ in Lemma \ref{Prop 9.1} translates to the $p$-adic side. In fact, consider the $p$-adic $L$-function of $\theta_{\varphi/K}$ of Theorem \ref{C1.2} and the $p$-adic $L$-function $L_p(-)$ constructed by Katz in \cite{katz1978p} and generalized by Hida and Tilouine in \cite{hida1993anti}. We obtain the following (see Theorem \ref{T9.4} for more details):

\begin{theorem}
Under the hypothesis of Theorem \ref{C1.2}, for all $\kappa$ we have
\begin{equation*}
L_p(\theta_{\varphi/K}^p,\kappa)=\frac{L_p(\varphi_{p-\rm{fin}}^c\kappa \chi_p)L_p(\varphi_{p-\rm{fin}}^c\kappa^c \lambda_p\chi_p)}{\Omega_p^{2k+2}},
\end{equation*}
where $\lambda_p,\chi_p$ are the $p$-adic avatars of the character $\lambda_K$ (see Lemma \ref{Prop 9.1}) and the adelic norm $|\cdot|_{\A_K}$ respectively, and $\Omega_p$ is the $p$-adic period in Theorem \ref{Katztheorem}.
\end{theorem}

\section*{Acknowledgements}

I would like to thank my PhD advisor Daniel Barrera for suggesting this topic to me, as well as for the many conversations we have had on the subject. Special thanks to Chris Williams for kindly answering a lot of questions. I also thank Samit Dasgupta, Antonio Cauchi, Guhan Venkat and Eduardo Friedman for helpful comments that led to an improvement of this article. Finally, I would like to thank the referee for a careful reading of this manuscript and for excellent suggestions on improving
this text.  This work was funded by the National Agency for Research and Development (ANID)/Scholarship Program/BECA DOCTORADO NACIONAL/2018 - 21180506 and by Anid Fondecyt Postdoctorado 3240129. Likewise, some research visits were supported by the project FONDECYT 11201025.

\section{\texorpdfstring{$C$}{Lg}-cuspidal Bianchi modular forms}
In this section, we recall basic properties of Bianchi modular forms and define a suitable vanishing condition on its constant terms called $C$-cuspidality. In Section \ref{L-C-cuspi} we study $L$-functions of $C$-cuspidal forms.

\subsection{Notations} \label{notations}

Let $p$ be a rational prime and fix throughout the paper embeddings $\iota_\infty:\overline{\Q}\hookrightarrow\C$ and $\iota_p:\overline{\Q}\hookrightarrow\overline{\Q}_p$, note that $\iota_p$ fixes a $p$-adic valuation $v_p$ on $\overline{\Q}_p$. Let $K$ be an imaginary quadratic field with discriminant $-D_K$, ring of integers $\mathcal{O}_K$ and different ideal $\mathcal{D}_K$ generated by $\delta=\sqrt{-D_K}$. Let $\gn=(p)\gm$ be an ideal of $\mathcal{O}_K$ with $\gm$ coprime to $(p)$. We denote by $id$ and $c$ the two embeddings of $K$ into $\C$ and identify them with $\mathrm{Hom}(K,\overline{\Q})$ under $\iota_\infty$. Henceforth, we write $(k,\ell)$ for $k \cdot id+ \ell\cdot c\in\Z[{id,c}]$, with $k,\ell\geqslant0$. Denote by $K_{\gq}$ the completion of $K$ with respect to the prime $\gq$ of $K$, $\mathcal{O}_{\gq}$ the ring of integers of $K_{\gq}$ and fix a uniformizer $\varpi_\gq$ at $\gq$. Denote the adele ring of $K$ by $\mathbb{A}_K=\C\times\A_K^f$ where $\mathbb{A}_K^f$ are the finite adeles. Furthermore, denote the class group of $K$ by $\mathrm{Cl}(K)$ and the class number of $K$ by $h$, and -once and for all- fix a set of representatives $I_1,..., I_h$ for $\mathrm{Cl}(K)$, with $I_1=\mathcal{O}_K$ and each $I_i$ for $2\leqslant i \leqslant h$ integral and prime, with each $I_i$ coprime to $\gn$ and $\mathcal{D}_K$.
Let $V_{n}(R)$ denote the space of homogeneous polynomials over a ring $R$ in two variables of degree $n\geqslant0$. Note that $V_{n}(\mathbb{C})$ is an irreducible complex right representation of $\mathrm{SU}_2(\mathbb{C})$, and denote it by $\rho_{n}$. 

For a general Hecke character $\psi$ of $K$ we denote by $\psi_{\infty}$, $\psi_f$ the restriction of $\psi$ to $\C^\times$, $\mathbb{A}_K^{f,\times}$, respectively. Likewise, for a prime $\gq$, we denote $\psi_{\gq}$ the restriction of $\psi$ to $K_{\gq}^{\times}$, and for an ideal $\gf\subset\roi_K$ we write $\psi_\gf:=\prod_{\gq|\gf}\psi_\gq$. 

\subsection{Bianchi modular forms} \label{background} 
We write $K_0(\gn)$ for the open compact subgroup of $\mathrm{GL_2}(\roi_K\otimes_\Z\widehat{\Z})$ of matrices $\smallmatrixx{*}{*}{0}{*}$ modulo $\gn$. Let $\varphi$ be a Hecke character, with infinity type $(-k,-\ell)$ and conductor dividing $\gn$. For $u_f=\smallmatrixx{a}{b}{c}{d}\in K_0(\gn)$ we set $\varphi_{\gn}(u_f)=\varphi_{\gn}(d)=\prod_{\gq|\gn}\varphi_{\gq}(d_{\gq})$.

\begin{definition}
We say a function $\mathcal{F}:\mathrm{GL_2}(\A_K) \rightarrow V_{k+\ell+2}(\mathbb{C})$ is a \textit{Bianchi modular form} of weight $(k,\ell)$, level $K_0(\gn)$ and central action $\varphi$ if it satisfies:

i) $\mathcal{F}$ is left-invariant under $\mathrm{GL_2}(K)$;

ii) $\mathcal{F}(zg)=\varphi(z)\mathcal{F}(g)$ for $z\in \mathbb{A}_K^\times \cong Z(\mathrm{GL_2}(\mathbb{A}_K))$, where $Z(G)$ denotes the center of the group $G$; 

iii) $\mathcal{F}(gu)=\varphi_\gn(u_f)\mathcal{F}(g)\rho_{k+\ell+2}(u_\infty)$ for $u=u_f\cdot u_\infty \in K_0(\gn)\times\mathrm{SU_2}(\mathbb{C})$;

iv) $\mathcal{F}$ is an eigenfunction of the operators $D_{id}$, $D_c$, with eigenvalues  
$k^2/2+k$ and $\ell^2/2+\ell$ respectively. Here $D_\sigma/4 $, for $\sigma\in\{id,c\}$, denotes a component of the Casimir operator in the Lie algebra $\mathfrak{sl}_2(\mathbb{C})\otimes_{\mathbb{R}}\mathbb{C}$ (see \cite[\S 2.4]{hida1994critical}), and we are considering $\mathcal{F}(g_{\infty}g_f)$ as a function of $g_\infty \in \mathrm{GL_2}(\mathbb{C}).$

The space of such functions will be denoted by $\mathcal{M}_{(k,\ell)}(K_0(\gn),\varphi)$. We say $\mathcal{F}$ is a \textit{cuspidal} Bianchi modular form if also satisfies:

v) for all $g \in \mathrm{GL_2}(\mathbb{A}_K)$
\begin{equation*}
\int_{K\backslash\mathbb{A}_K}\mathcal{F}\left(\matrixx{1}{u}{0}{1}g\right)du=0,
\end{equation*}
where $du$ is the Lebesgue measure on $\mathbb{A}_K$. The space of such functions will be denoted by $S_{(k,\ell)}(K_0(\gn),\varphi)$.
\label{D1.2}
\end{definition}

\begin{remark}
From \cite[Cor 2.2]{hida1994critical}) we have $S_{(k,\ell)}(K_0(\gn),\varphi)=0$ if $k\neq \ell$ i.e., all non-trivial cuspidal Bianchi modular forms have parallel weight $(k,k)$.
\end{remark}

Recall the set of representatives $I_1,..., I_h$ for $\mathrm{Cl}(K)$ fixed in Section \ref{notations} and denote by $\varpi_i$ their corresponding fixed uniformizer. Set $g_i=\smallmatrixx{t_i}{0}{0}{1}$ where $t_1=1$ and for each $i\geqslant2$, define $t_i=(1,...,1,\varpi_i,1,...)\in\mathbb{A}_K^{\times}$. Since $\mathrm{GL_2}(\mathbb{A}_K)=\coprod_{i=1}^{h}\mathrm{GL_2}(K)\cdot g_i\cdot [\mathrm{GL_2}(\mathbb{C})\times K_0(\gn)]$, a Bianchi modular form $\mathcal{F}\in\mathcal{M}_{(k,\ell)}(K_0(\gn),\varphi)$, descends to a collection of $h$ functions $F^i: \mathrm{GL_2}(\mathbb{C})\longrightarrow V_{k+\ell+2}(\mathbb{C})$ via $F^i(g):=\mathcal{F}(g_ig)$.

Since $\mathrm {GL_2}(\mathbb{C})/\C^\times\mathrm{SU_2}(\mathbb{C})$ is isomorphic to the hyperbolic 3-space $\mathcal{H}_3:=\mathbb{C}\times\mathbb{R}_{>0}$, we can descend further using properties ii) and iii) in Definition \ref{D1.2} to obtain $h$ functions $f^i: \mathcal{H}_3\longrightarrow V_{k+\ell+2}(\mathbb{C})$ via $f^i(z,t):= t^{-1}F^i\smallmatrixx{t}{z}{0}{1}$. 

Let $\gamma=\smallmatrixx{a}{b}{c}{d}\in
\Gamma_0^i(\gn):=\mathrm{SL_2}(K)\cap g_i K_0(\gn) g_i^{-1}\mathrm{GL_2}(\mathbb{C})$, then each $f^i$ satisfies the automorphy condition 
\begin{equation}
f^i(\gamma\cdot(z,t))=\varphi_\gn(d)^{-1}f^i(z,t)\rho_{k+\ell+2}(J(\gamma;(z,t))),
\label{e2.4}
\end{equation}
where $\gamma\cdot(z,t):=\left(\frac{(az+b)\overline{(cz+d)}+a\overline{c}|t|^2}{|cz+d|^2+|ct|^2}, \frac{|ad-bc|t}{|cz+d|^2+|ct|^2} \right)$ is the standard action of $\mathrm {GL_2}(\mathbb{C})$ on $\mathcal{H}_3$, in which $|\cdot|$ denotes the norm in $\C$; and $J(\gamma;(z,t)):=\smallmatrixx{cz+d}{\overline{ct}}{-ct}{\overline{cz+d}}$. Thus $f^i\in \mathcal{M}_{(k,\ell)}(\Gamma_0^i(\gn),\varphi_\gn^{-1})$, the space of \textit{Bianchi modular forms} on $\mathcal{H}_3$ satisfying the automorphy condition (\ref{e2.4}). If $\mathcal{F}\in S_{(k,k)}(K_0(\gn),\varphi)$ is cuspidal then we say $f^i$ is a cuspidal Bianchi modular form and the space of such forms is denoted by $S_{(k,k)}(\Gamma_0^i(\gn),\varphi_\gn^{-1})$.

\begin{definition}\label{D2.3}
Let $\gamma\in \mathrm{GL_2}(\mathbb{C})$, for $f^i\in \mathcal{M}_{(k,\ell)}(\Gamma_0^i(\mathfrak{n}),\varphi_\gn^{-1})$, define the function $f^i|_\gamma:\mathcal{H}_3\rightarrow V_{k+\ell+2}(\C)$ by
\begin{equation}
(f^i|_\gamma)(z,t)=det(\gamma)^{-k/2}\overline{det(\gamma)}^{-\ell/2}f^i(\gamma \cdot(z,t))  \rho_{k+\ell+2}^{-1}\left(J\left(\frac{\gamma}{\sqrt{det(\gamma)}} ;(z,t)\right)\right).
\end{equation}
\end{definition}

\begin{remark}
Note $f^i\in \mathcal{M}_{(k,\ell)}(\Gamma_0^i(\mathfrak{n}),\varphi_\gn^{-1})$ satisfies:

i) $f^i|_g(0,1)=F^i(g)$ for $g\in\mathrm{GL}_2(\mathbb{C})$, in particular, if $g=\smallmatrixx{t}{z}{0}{1}$ with $(z,t)\in\mathcal{H}_3$, we recover $f^i(z,t)= t^{-1}F^i\smallmatrixx{t}{z}{0}{1}$.

ii) $\left(f^i|_\gamma\right)(z,t)=\varphi_\gn(d)^{-1}f^i(z,t)$ for $\gamma=\smallmatrixx{a}{b}{c}{d} \in \Gamma_0^i(\mathfrak{n})$. 
\end{remark}

\subsection{Fourier expansion and cuspidal conditions}{\label{Fourier}}

Recall that $k,\ell\geqslant0$, consider the set $R_{k,\ell}=\{(\pm(k+1),\pm(\ell+1))\}$ and for $r=(r_1,r_2)\in R_{k,\ell}$ define 
\begin{equation*}
h(r)=\frac{k+\ell+2}{2}+\frac{r_2-r_1}{2},\;\;\text{and}\;\;
g(r)=\left(\frac{1}{2}[r_1-(k+1)],\frac{1}{2}[r_2-(\ell+1)]\right).    
\end{equation*}

Let $\mathcal{F}$ be a Bianchi modular form of weight $(k,\ell)$, level $K_0(\gn)$ and central action $\varphi$, then it has a Fourier expansion given by (see \cite[Thm 6.7]{hida1994critical}):
\begin{align}\label{e.2.1}
\mathcal{F}\left[\matrixx{t}{z}{0}{1}\right]=|t|_{\mathbb{A}_K}\bigg[&\sum_{r\in R_{k,\ell}}t_\infty^{g(r)}\binom{k+\ell+2}{h(r)}c_r(t\mathcal{D}_K,\mathcal{F})X^{k+\ell+2-h(r)}Y^{h(r)}\\
&+ \sum_{\alpha \in K^\times} c(\alpha t \mathcal{D}_K,\mathcal{F})W(\alpha t_\infty) \textbf{e}_K(\alpha z)\bigg], \nonumber 
\end{align}
where
\begin{itemize}
\item[i)] The Fourier coefficients $c(\cdot,\mathcal{F})$, and $c_r(\cdot,\mathcal{F})$ for each $r\in R_{k,\ell}$, are functions on the fractional ideals of $K$ that vanish outside the integral ideals.
\item[ii)] $\textbf{e}_K$ is an additive character of $K\backslash \mathbb{A}_K$ defined by
\begin{equation*}
\textbf{e}_K=\left( \prod_{\gq \mathrm{prime}} (\textbf{e}_q \circ \mathrm{Tr}_{K_\gq/\mathbb{Q}_q})\right)\cdot (\textbf{e}_\infty \circ \mathrm{Tr}_{\mathbb{C}/\mathbb{R}}), 
\end{equation*}
for 
\begin{equation*}
\textbf{e}_q\left( \sum_j d_jq^j \right)=e^{-2\pi i \sum_{j<0} d_jq^j} \;\;\;  \mathrm{and} \;\;\; \textbf{e}_\infty(r)=e^{2\pi ir};    
\end{equation*}
\item[iii)] $W:\mathbb{C}^\times \rightarrow V_{k+\ell+2}(\mathbb{C})$ is the Whittaker function
\begin{equation*}
W(s):= \sum_{n=0}^{k+\ell+2} \binom{k+\ell+2}{n}\left( \frac{s}{i|s|} \right)^{\ell+1-n} K_{n-\ell-1}(4 \pi |s|)X^{k+\ell+2-n}Y^n,    
\end{equation*}
where $K_n(x)$ is the modified Bessel function of order $n$ as in \cite[\S 6]{hida1994critical}.
\end{itemize}

\begin{remark}\label{Remark1.8}
i) Note that $W(s)$ is not symmetric in $k$ and $\ell$, this comes from the definition of the Whittaker function in \cite[(6.1)]{hida1994critical} after fixing the weight $(k,\ell)$.

\noindent ii) Let $\mathcal{F}=\sum_{n=0}^{k+\ell+2}\mathcal{F}_nX^{k+\ell+2-n}Y^n$ be a Bianchi modular form, then by (\ref{e.2.1}), the constant term in the Fourier expansion of $\mathcal{F}_n$ is trivial if $n\notin\{h(r):r\in R_{(k,\ell)}\}=\{0,k+1,\ell+1,k+\ell+2\}$. This fact will be used in the proof of Proposition \ref{basechangequasi}.
\end{remark}

The Fourier expansion of $\mathcal{F}$ descends to $\mathcal{H}_3$ by 
\begin{equation}\label{e.3.2} 
f^j\left((z,t); \matrixxx{X}{Y} \right)=\sum_{n=0}^{k+\ell+2}f_n^j(z,t) X^{k+\ell+2-n}Y^{n},
\end{equation}
where
\begin{align*}\label{e.3.2}
f_n^j(z,t)&=\left[t^{\frac{r_1+r_2-k-\ell}{2}}\binom{k+\ell+2}{n}c_r(t_j\mathcal{D}_K)\right]\delta_{n,h(r)}\\
&+|t_j|_f t\binom{k+\ell+2}{n}\sum_{\alpha\in K^\times} \left[ c(\alpha t_j\mathcal{D}_K)\left(\frac{\alpha}{i|\alpha|}\right)^{\ell+1-n} K_{n-\ell-1}(4\pi|\alpha|t)e^{2\pi i (\alpha z + \overline{\alpha z})}\right]    
\end{align*}
with $r=(r_1,r_2)\in R_{k,\ell}$ as above and $\delta_{n,h(r)}=1$ if $n=h(r)$ and $\delta_{n,h(r)}=0$ otherwise. Also note that to ease notation, we have written $c_r(t_j\mathcal{D}_K)$ (resp.  $c(\alpha t_j \mathcal{D}_K)$) instead $c_r(t_j\mathcal{D}_K,\mathcal{F})$ (resp. $c(\alpha t_j\mathcal{D}_K,\mathcal{F})$).

For each $i=1,...,h$, equation (\ref{e.3.2}) may be thought of as the Fourier expansion of $f^i$ at the cusp of infinity, which by Remark \ref{Remark1.8}, satisfies that the constant term in the Fourier expansion of $f^i_n$ is trivial if $n\notin\{0,k+1,\ell+1,k+\ell+2\}$.  

We must consider Fourier expansions at all the ``$K$-rational'' cusps $\mathbb{P}^1(K)=K\cup\{\infty\}$, for this, let $\sigma\in\mathrm{GL}_2(K)$ sending $\infty$ to the cusp $s$. For each $i=1,..., h$, since $f^i\in \mathcal{M}_{(k,\ell)}(\Gamma_0^i(\gn),\varphi_\gn^{-1})$ then $f^i|_{\sigma}\in \mathcal{M}_{(k,\ell)}(\sigma^{-1}\Gamma_0^i(\gn)\sigma,\varphi_{\gn}^{-1})$ and hence $f^i|_{\sigma}$ has a Fourier expansion as in (\ref{e.3.2}).

\begin{definition}
We say that $f^i$ \textit{vanishes} at the cusp $s$ if $(f^i|_{\sigma})_n$ has trivial constant term for $0\leqslant n\leqslant k+\ell+2$, and \textit{quasi-vanishes} at the cusp $s$ if $(f^i|_{\sigma})_n$ has trivial constant term for $1\leqslant n\leqslant k+\ell+1$.
\end{definition}

\begin{remark}

i) The property of vanishing and quasi-vanishing at the cusp $s$ are well-defined, i.e., are independent of the choice of $\sigma$; for the vanishing case see \cite[\S 6.2.2]{bygott1998modular} and note that the same argument works for quasi-vanishing.

ii) Let $\mathcal{F}\in S_{(k,k)}(K_0(\gn),\varphi)$ be a cuspidal Bianchi modular form, then the cuspidal condition v) in Definition \ref{D1.2} is equivalent to the vanishing of $f^i$ at all cusps for each $0\leqslant i \leqslant h$ (see \cite[Prop 3.2]{zhao1993certain}).
\end{remark}

Recall that $\gn=(p)\gm$ with $\gm$ coprime to $(p)$ and define for each $i=1,..., h$ the set of cusps
\begin{equation*}
C_i:=\Gamma_0^i(\gm)\infty\cup\Gamma_0^i(\gm)0.
\end{equation*}

Since $\Gamma_0^i(\gm)=\{\smallmatrixx{a}{b}{c}{d}\in\mathrm{SL}_2(K): b\in I_i,\;c\in\gm I_i^{-1}\}$, we have
$C_i\subset\mathbb{P}^1(K)$ contains $\infty$ and elements $\frac{x}{y}\in K$ with $x\in I_i$ and either $y\in\gm$ or $y\in(\mathcal{O}_K/\gm)^\times$. 

\begin{definition}{\label{Def 2.9}}
We say that $f^{i}$ is \textit{$C_i$-cuspidal} if quasi-vanishes at all cusps in $C_i$.
\end{definition}

The previous definition of $C_i$-cuspidality differs from the one given in \cite{bellaiche2015p} for modular forms. We are not asking for the vanishing of $f^i$ at the cusps $C_i$, instead, we just need quasi-vanishing, i.e., we do not care about the vanishing of the functions $f^i_0$ and $f^i_{k+\ell+2}$ at the cusps $C_i$. The motivation for considering quasi-vanishing instead of vanishing will become clear in Proposition \ref{T4.6}.

To state $C_i$-cuspidality for all $i$ as a property of $\mathcal{F}$, we write $C:=(C_1,...,C_h)$ and introduce the notion of $C$-cuspidality.

\begin{definition}{\label{C-cusp}}
We say that $\mathcal{F}$ is \textit{C-cuspidal} if $f^i$ is $C_i$-cuspidal for $i=1,..., h$.
\end{definition}

\begin{remark}\label{rem1.11}
Note that for Bianchi modular forms with level at $p$ we have 
\begin{equation*}
\{\text{cuspidal}\}\subset\{C\text{-cuspidal}\}\subset\{\text{Bianchi modular forms}\}.   
\end{equation*}
\end{remark}

We now define Hecke operators acting on $\mathcal{M}_{(k,\ell)}(K_0(\gn),\varphi)$. For this, recall our fixed representatives $I_1,...,I_h$ for the class group, and note that for a prime $\gq$ of $K$ and each $i \in \{1,...,h\}$ there is a unique $j_i\in\{1,...,h\}$ such that $\gq I_i =(\alpha_i)I_{j_i}$, for $\alpha_i\in K$. Then we define the Hecke operators $T_\gq$ acting on Bianchi modular forms $\mathcal{F}=(f^1,...,f^h)$ by double cosets on each component by
\begin{equation}{\label{e.1.19}}
\mathcal{F}|_{T_\gq}:=\left(\alpha_1^k\overline{\alpha_1}^\ell f^{j_1}|_{\left[\Gamma_0^{j_1}(\gn)\smallmatrixx{1}{0}{0}{\alpha_1}\Gamma_0^1(\gn)\right]},...,\alpha_h^k\overline{\alpha_h}^\ell f^{j_h}|_{\left[\Gamma_0^{j_h}(\gn)\smallmatrixx{1}{0}{0}{\alpha_h}\Gamma_0^h(\gn)\right]}\right).
\end{equation}
If $\gq|\gn$ we denote the Hecke operator by $U_\gq$. Note that when $K$ has class number one and $k=\ell$, the previous action is the same from \cite[\S 2.4]{palacios}.

An eigenform is a Bianchi modular form that is a simultaneous eigenvector for all the Hecke operators, in this case, the eigenvalues and Fourier coefficients of the eigenform at each prime $\gq$ coincides.

\begin{definition}
Let $\mathcal{H}_{\gn,p}$ denote the $\Q$-algebra generated by the Hecke operators $\{T_\gq : (\gq,\gn)=1\}$ and $\{U_\gp :\gp|p\}$.
\end{definition}

\begin{lemma}\label{l: cuspides}
Let $\gn=(p)\gm$ be an ideal with $(\gm,(p))=1$, $\gq$ be a prime ideal not dividing $\gm$ and $C_i=\Gamma_0^i(\gm)\infty\cup\Gamma_0^i(\gm)0$ for each $i=1,..., h$, then:

i) $\Gamma_0^i(\gn)$ stabilizes $C_i$.

ii) $\smallmatrixx{1}{0}{0}{\alpha_i}\cdot c_i\in C_{j_i}$ for all $c_i\in C_i$, where $\gq I_i =(\alpha_i)I_{j_i}$.    
\end{lemma}

\begin{proof}
For i) just note that $\Gamma_0^i(\gm)$ (and hence its subgroup $\Gamma_0^i(\gn)$) stabilizes $C_i$.

To prove ii), since $(\gq,\gm)=1$ and $(I_i,\gm)=1$, there exists $y_\gq\in\gq$ and $y_i\in I_i$, such that $y_\gq,y_i\in(\mathcal{O}_K/\gm)^\times$. By the identity $\gq I_i =(\alpha_i)I_{j_i}$, there exist an element $t_{j_i}\in I_{j_i}$ such that
\begin{equation}{\label{equat1.20}}
\alpha_i=\frac{y_\gq y_i}{t_{j_i}}.  
\end{equation}
Let $c_i=x/y$, then $x\in I_i$ and either $y\in\gm$ or $y\in(\mathcal{O}_K/\gm)^\times$, then we have
\begin{equation*}
\smallmatrixx{1}{0}{0}{\alpha_i}\cdot \frac{x}{y}= \frac{x}{\alpha_iy}=\frac{t_{j_i}x} {y_\gq y_iy} 
\end{equation*}
with $t_{j_i}x\in I_{j_i}$ and either $y_\gq y_iy\in\gm$ or $y_\gq y_iy\in(\mathcal{O}_K/\gm)^\times$ then $\smallmatrixx{1}{0}{0}{\alpha_i}\cdot c_i \in C_{j_i}$. 
\end{proof}

\begin{proposition}
Let $\mathcal{F}\in\mathcal{M}_{(k,\ell)}(K_0(\gn),\varphi)$ be a $C$-cuspidal Bianchi modular form, with $\gn=(p)\gm$ and $(\gm,(p))=1$; then $\mathcal{H}_{\gn,p}$ acts on $C$-cuspidal forms.
\end{proposition}

\begin{proof}
By (\ref{e.1.19}) we have to show that for all prime $\gq\nmid\gm$ with $\gq I_i =(\alpha_i)I_{j_i}$ the function $f^{j_i}|{\left[\Gamma_0^{j_i}(\gn)\smallmatrixx{1}{0}{0}{\alpha_i}\Gamma_0^i(\gn)\right]}$ is $C_{i}$-cuspidal. For this, take $s_i\in C_i$, $\sigma_{s_i}\in\mathrm{GL}_2(K)$ such that $\sigma_{s_i}\cdot\infty=s_i$ and $\gamma\in\Gamma_0^{j_i}(\gn)\smallmatrixx{1}{0}{0}{\alpha_i}\Gamma_0^i(\gn)$. We have to show that the constant term of  $(f^{j_i}|_{\gamma\sigma_{s_i}})_n$ vanishes for $1\leqslant n\leqslant k+\ell+1$.

Write $\gamma=\gamma_{j_i}\smallmatrixx{1}{0}{0}{\alpha_i}\gamma_i$ with $\gamma_j\in\Gamma_0^j(\gn)$, then using Lemma \ref{l: cuspides} we have $\gamma_i\cdot s_i=s_i'\in C_i$ by part i); $\smallmatrixx{1}{0}{0}{\alpha_i}\cdot s_i'=s_{j_i}\in C_{j_i}$ by part ii); and $\gamma_{j_i}\cdot s_{j_i}=s_{j_i}'\in C_{j_i}$ by part i). Then $\gamma\sigma_{s_i}\cdot\infty=\gamma_{j_i}\smallmatrixx{1}{0}{0}{\alpha_i}\gamma_i\sigma_{s_i}\cdot\infty=\gamma_{j_i}\smallmatrixx{1}{0}{0}{\alpha_i}\gamma_i\cdot s_i=\gamma_{j_i}\smallmatrixx{1}{0}{0}{\alpha_i}\cdot s_i'=\gamma_{j_i}\cdot s_{j_i}=s_{j_i}'\in C_{j_i}$.

The result now follows since $f^{j_i}$ is $C_{j_i}$-cuspidal and $\gamma\sigma_{s_i}\cdot\infty\in C_{j_i}$.
\end{proof}

\subsection{\texorpdfstring{$L$}{Lg}-function of \texorpdfstring{$C$}{Lg}-cuspidal Bianchi modular forms}{\label{L-C-cuspi}}

Henceforth, we will consider $\psi$ to be a Hecke character over $K$ of conductor $\gf$ with $(\gf,I_i)=1$ for each $i$. For each ideal $\mathfrak{a}=\prod_{\mathfrak{q}|\mathfrak{a}}\mathfrak{q}^{n_{\mathfrak{q}}}$ coprime to $\mathfrak{f}$, we define $\psi(\mathfrak{a})=\prod_{\mathfrak{q}|\mathfrak{a}}\psi_{\mathfrak{q}}(\varpi_{\mathfrak{q}})^{n_{\mathfrak{q}}}$ and $\psi(\mathfrak{a})=0$ otherwise. In an abuse of notation, we write $\psi$ for both the idelic Hecke character and the function it determines on ideals.

\begin{definition}
The $L$-function of a Bianchi modular form $\mathcal{F}$ twisted by $\psi$ is defined by
\begin{equation*}
L(\mathcal{F},\psi,s):=\sum_{\substack{0\neq \mathfrak{a}\subset\mathcal{O}_K\\ (\mathfrak{a},\gf)=1}} c(\mathfrak{a},\mathcal{F})\psi(\mathfrak{a})N(\mathfrak{a})^{-s} \;\;\;(s\in\mathbb{C}). 
\end{equation*}
\end{definition}

Let $f_1,..., f_h$ be the $h$ Bianchi modular forms corresponding to $\mathcal{F}$ with respect to $I_1,.., I_h$, then $L$-function of $\mathcal{F}$ can be written as
\begin{equation}{\label{sumLfunctions}}
L(\mathcal{F},\psi,s)= L^1(\mathcal{F},\psi,s)+\cdot\cdot\cdot+ L^h(\mathcal{F},\psi,s),  
\end{equation}
where 
\begin{equation*}
L^i(\mathcal{F},\psi,s)=L(f^i,\psi,s):=w_K^{-1}\sum_{\alpha\in K^\times} c(\alpha\delta I_i,\mathcal{F})\psi(\alpha\delta I_i)N(\alpha\delta I_i)^{-s}  \end{equation*}
with $w_K=|\mathcal{O}_K^\times|$.

\begin{remark}
In \cite[Chap. II]{weil1971dirichlet} it is proved that the $L$-functions of $\mathcal{F}$ and each $f^i$ converge absolutely on some right-half plane. 
\end{remark}

Note that for a complex number $s$, we have $L(\mathcal{F},\psi,s)=L(\mathcal{F},\psi|\cdot|_{\A_K}^{s-1},1)$ which allows us following \cite[\S 2.6]{chris2017}, consider the $L$-function of a $C$-cuspidal Bianchi modular form as a function on Hecke characters $\psi$ by setting 
$$L(\mathcal{F},\psi):=L(\mathcal{F},\psi,1).$$

We now complete the $L$-function by adding Deligne's $\Gamma$-factors at infinity. If $\psi$ has infinity type $(u,v)\in\C^2$ we define the  \textit{renormalized} $L$-function of $\mathcal{F}$ by:
\begin{equation}{\label{gammafactors}}
\Lambda(\mathcal{F},\psi):=\frac{\Gamma(u+1)\Gamma(v+1)}{(2\pi i)^{u+1}(2\pi i)^{v+1}}L(\mathcal{F},\psi). 
\end{equation}

The $L$-function of $C$-cuspidal Bianchi modular forms can be written as a finite sum of Mellin transforms similarly as in \cite[\S 2.6]{chris2017}. Before stating the integral expression of the $L$-function, we recall some facts about Gauss sums and define suitable coefficients to link $\Lambda(\mathcal{F},\psi)$ with partial modular symbols in Section \ref{partialsymbolsattached}.

Let $\psi$ be a Hecke character of $K$ with conductor $\mathfrak{f}$, then we define the Gauss sum of $\psi$ to be
\begin{equation*}
\tau(\psi)=\psi_\infty(\delta)\sum_{\substack{[a]\in\mathfrak{f}^{-1}/\mathcal{O}_K \\ ((a)\gf,\gf)=1}}\psi_\gf(a) e^{2\pi i \mathrm{Tr}_{K/\mathbb{Q}}(a/\delta)}.
\end{equation*}

Note that our Gauss sum $\tau$ is related with the Gauss sum $\tilde{\tau}$ defined in \cite[\S2.6]{chris2017}. More precisely, by denoting $x_\gf$ the idele associated to the ideal $\gf$ as in \cite[\S2.6]{chris2017} we have that $\psi(a\gf)^{-1}\psi_\infty(a)^{-1}=\psi(x_\gf)^{-1}\psi_\gf(ax_\gf)$ for $a\in\gf^{-1}/\roi_K$ with $(a\gf,\gf)=1$, then we deduce that
$$\tau(\psi)=\psi(x_\gf)\psi_\gf(x_\gf)^{-1}\tilde{\tau}(\psi^{-1}).$$

\begin{remark}{\label{Gauss}}

i) In the same way as in \cite[Prop 1.7]{chris2017}, for all $c\in\mathcal{O}_K$, we have 
\begin{equation*}
\frac{\psi_\infty(\delta)}{\tau(\psi)}\sum_{\substack{[a]\in\mathfrak{f}^{-1}/\mathcal{O}_K \\((a)\mathfrak{f},\mathfrak{f})=1}}\psi_\gf(a) e^{2\pi i \mathrm{Tr}_{K/\mathbb{Q}}(ac/\delta)}=\begin{cases}
\psi_\gf(c)^{-1}&: ((c)\gf,\gf)=1,\\
0 &:\rm{otherwise}.
\end{cases}    
\end{equation*}

ii) Suppose $\gf$ coprime to $\gm$ and let $i\in\{1,...,h\}$, then we can choose a representative $a$ for each $[a]\in\gf^{-1}/\mathcal{O}_K$, such that $a\in C_i$: let $I_{k_i}$ be the unique ideal such that $\gf I_{k_i} =(\alpha_{k_i})I_i$ for $\alpha_{k_i}\in K$, then we put $a=d_b/\alpha_{k_i}$ where $d_b \in I_{k_i}$ and $d_b\equiv b$ (mod $\gf$). Note that $\alpha_{k_i}^{-1}\in\gf^{-1}I_{k_i}^{-1}I_i\subset \gf^{-1}I_{k_i}^{-1}$, so in particular, as $b$ ranges over all classes of $(\mathcal{O}_K/\gf)^\times$ and as $d_b \in I_{k_i}$, we see that $d_b/\alpha_i$ ranges over a full set of coset representatives $[a]$ for $\gf^{-1}/\mathcal{O}_K$ with
$(a)\gf$ coprime to $\gf$. Similarly to (\ref{equat1.20}), we have that
$\alpha_{k_i}=y_\gf y_{k_i}/t_i$ with
$y_\gf\in(\mathcal{O}_K/\gm)^\times$, $y_{k_i}\in(\mathcal{O}_K/\gm)^\times$ and $t_i\in I_i$. Since $y_\gf y_{k_i} \in(\mathcal{O}_K/\gm)^\times$, then we obtain
$a=d_b/\alpha_{k_i}=d_b t_i/y_\gf y_{k_i}\in C_i$.
\end{remark}

\begin{definition}{\label{factor}}
Let $\mathcal{F}$ be a $C$-cuspidal form of weight $(k,\ell)$, for each $1 \leqslant i \leqslant h$, $0\leqslant q \leqslant k, 0\leqslant r \leqslant \ell$ and $a\in C_i$ we define
$$c_{q,r}^{i}(a)=2\binom{k+\ell+2}{\ell+q-r+1}^{-1} (-1)^{\ell-r+1}\int_{0}^{\infty}t^{q+r}f_{\ell+q-r+1}^{i}(a,t)dt.$$
\end{definition}

\begin{remark}
The integral in the above definition converges since $1\leqslant \ell+q-r+1 \leqslant k+\ell+1$ and $f^i$ is $C_i$-cuspidal. Also, note that $c_{q,r}^{i}$ is not symmetric in $k$ and $\ell$, this comes from part i) of Remark \ref{Remark1.8}. 
\end{remark}

\begin{proposition}\label{T4.6}
Let $\mathcal{F}\in\mathcal{M}_{(k,\ell)}(K_0(\gn))$ be a $C$-cuspidal Bianchi modular form with $\gn=(p)\gm$ and $\gm$ coprime to $(p)$, then for a Hecke character $\psi$ of $K$ of conductor $\gf$ coprime to $\gm$ and infinity type $(q,r)$ satisfying $0\leqslant q \leqslant k, 0\leqslant r \leqslant \ell$, we have
\begin{equation*}
\Lambda(\mathcal{F},\psi)=\frac{(-1)^{\ell+q+r}2}{D_Kw_K\tau(\psi)}\left[\sum_{j=1}^h\psi(t_j)\sum_{\substack{[a]\in\gf^{-1}/\mathcal{O}_K \\ ((a)\gf,\gf)=1,\;a\in C_j}}\psi_\gf(a) c_{q,r}^{j}(a)\right].
\end{equation*}
\end{proposition}

\begin{proof}
Let $f^1,...,f^h$ be the Bianchi modular forms corresponding to $\mathcal{F}$. Expanding the $L$-function of $f^j$ we have
\begin{equation}\label{primeraexpre}
L(f^j,\psi,1)=\frac{\psi(t_j)|t_j|_f}{w_K}\sum_{\substack{\alpha\in K^\times\\ (\alpha\delta I_j,\gf)=1}} c(\alpha\delta I_j,\mathcal{F})\psi_\infty(\alpha\delta)^{-1}\psi_\gf(\alpha\delta)^{-1}|\alpha\delta|^{-2}.
\end{equation}

We now proceed to rewrite $\psi_\infty(\delta)^{-1}\psi_\gf(\alpha\delta)^{-1}|\delta|^{-2}$ in terms of a Gauss sum and $\psi_\infty(\alpha)^{-1}|\alpha|^{-2}$ as an integral.

By parts i) and ii) of Remark \ref{Gauss} we have
\begin{equation}\label{segundaexpre}
\psi_\infty(\delta)^{-1}\psi_\gf(\alpha\delta)^{-1}|\delta|^{-2}=\frac{1}{D_K\tau(\psi)}\sum_{\substack{[a]\in\gf^{-1}/\mathcal{O}_K \\ ((a)\gf,\gf)=1,\;a\in C_j}}\psi_\gf(a)e^{2\pi i \mathrm{Tr}_{K/\mathbb{Q}}(a\alpha)}.    
\end{equation}

Likewise, we obtain 
\begin{align}\label{terceraexpre}
\psi_\infty(\alpha)^{-1}|\alpha|^{-2}&=\left(\frac{\alpha}{|\alpha|}\right)^{-q+r}|\alpha|^{-q-r-2}    \\
&=\left(\frac{\alpha}{|\alpha|}\right)^{-q+r} \frac{4(2\pi)^{q+r+2}}{\Gamma(q+1)\Gamma(r+1)}\int_0^\infty t^{q+r+1}K_{q-r}(4\pi|\alpha|t)dt, \nonumber
\end{align}
where the last integral comes by setting $\lambda=4\pi|\alpha|$, $a=q+r+2$ and $b=q-r$ in the standard integral (see \cite[\S 7]{hida1994critical}) 
$$\int_0^\infty t^{a-1}K_b(\lambda t)dt=\lambda^{-a}2^{a-2}\Gamma\left(\frac{a+b}{2}\right)\Gamma\left(\frac{a-b}{2}\right).$$

Substituting (\ref{segundaexpre}) and (\ref{terceraexpre}) in (\ref{primeraexpre}) and rearranging, we obtain the Fourier expansion of $f_{\ell+q-r+1}^j$, more precisely we have 
\begin{equation}\label{secondexpre}
L(f^j,\psi,1)=A(j,\psi)\left[\sum_{\substack{[a]\in\gf^{-1}/\mathcal{O}_K \\ ((a)\gf,\gf)=1,\;a\in C_j}}\psi_\gf(a)\int_{0}^{\infty}t^{q+r}f_{\ell+q-r+1}^j(a,t)dt\right],
\end{equation}
where 
\begin{equation*}
A(j,\psi)=\frac{\psi(t_j)(-1)^{q+1} 4(2\pi i)^{q+r+2}\binom{k+\ell+2}{\ell+q-r+1}^{-1}}{D_Kw_K\Gamma(q+1)\Gamma(r+1)\tau(\psi)},   
\end{equation*}
and the integral converges by the $C$-cuspidality of $\mathcal{F}$, since the conditions on the infinity type $(q,r)$ of $\psi$ implies $1\leqslant \ell+q-r+1 \leqslant k+\ell+1$.

The result follows by recalling that $\Lambda(\mathcal{F},\psi)=\frac{\Gamma(q+1)\Gamma(r+1)}{(2\pi i)^{q+r+2}}\sum_{j=1}^hL(f^j,\psi,1)$.
\end{proof}

We finish this section with some remarks about algebraicity of $L$-values of $C$-cuspidal Bianchi modular forms.

For cuspidal Bianchi modular forms, the ``critical'' values of this $L$-function can be controlled. We have the following theorem (see \cite[Thm 8.1]{hida1994critical}): 

\begin{theorem}
Let $\mathcal{F}$ be a cuspidal Bianchi eigenform of weight $(k,k)$, there exists a period $\Omega_\mathcal{F}\in\mathbb{C}^\times$ and a number field $E$ such that, if $\psi$ is a Hecke character of infinity type $(q,r)$ satisfying $0\leqslant q \leqslant k, 0\leqslant r \leqslant k$, we have
\begin{equation*}
\frac{\Lambda(\mathcal{F},\psi)}{\Omega_\mathcal{F}}\in E(\psi),  
\end{equation*}
where $E(\psi)\subset\overline{\mathbb{Q}}$ is the extension of $E$ generated by the values of $\psi_f$.
\end{theorem}

\begin{remark}{\label{isomorphism}}
In the non-cuspidal case, depending on the Bianchi modular forms we are interested, we can prove algebraicity of critical $L$-values (see Proposition \ref{period base change} for an example). For the construction of the $p$-adic $L$-function, we are interested in view $p$-adically the critical $L$-values. For this, in Section \ref{S4} we will use an isomorphism between $\C$ and $\overline{\Q}_p$.   
\end{remark}

\section{Partial Bianchi modular symbols}{\label{S3}}

In this section, we introduce partial Bianchi modular symbols. These are algebraic analogues of $C$-cuspidal Bianchi modular forms that are easier to study $p$-adically. In Section \ref{partialsymbolsattached}, we attach partial symbols to $C$-cuspidal forms and link them with $L$-values.

\subsection{Partial modular symbols}

Let $\Gamma$ be a discrete subgroup of $\mathrm{SL}_2(K)$ and let $\mathcal{C}$ be a non-empty $\Gamma$-invariant subset of $\mathbb{P}^1(K)$.

We denote by $\Delta_{\mathcal{C}}$ the abelian group of divisors on $\mathcal{C}$, i.e.,
\begin{equation*}
\Delta_{\mathcal{C}}=\Z[\mathcal{C}]=\left\{ \sum_{c\in \mathcal{C}}n_c\{c\}: n_c\in\mathbb{Z}, \; n_c=0 \;\text{for}\;\text{almost}\;\text{all}\;c \right\}
\end{equation*}
and by $\Delta_{\mathcal{C}}^0$ the subgroup of divisors of degree 0 (i.e., such that $\sum_{c\in\mathcal{C}}n_c=0$). Note that $\Delta_{\mathcal{C}}^0$ has a left action by the group $\Gamma$ by fractional linear transformations on $\mathcal{C}$. 

Let $V$ be a right $\Gamma$-module, we provide the space $\mathrm{Hom}(\Delta_{\mathcal{C}}^0,V)$ with a right $\Gamma$-action by setting 
\begin{equation*}
\phi_{|\gamma}(D)=\phi(\gamma \cdot D)_{|\gamma}. 
\end{equation*}
\begin{definition}{\label{abstract partial}}
We define the space of partial modular symbols on $\mathcal{C}$ for $\Gamma$ with values in $V$, to be the space $\mathrm{Symb}_{\Gamma,\mathcal{C}}(V):= \mathrm{Hom}_\Gamma(\Delta_{\mathcal{C}}^0,V)$ of $\Gamma$-invariant maps from $\Delta_{\mathcal{C}}^0$ to $V$. 
\end{definition}

\begin{remark}
When $\mathcal{C}=\mathbb{P}^1(K)$ we drop $\mathcal{C}$ from the notation and call $\mathrm{Symb}_{\Gamma}(V)$ the space of \textit{modular symbols} for $\Gamma$ with values in $V$ recovering \cite[Def 2.3]{chris2017}.
\end{remark}

Recall from Section \ref{background} the group $K_0(\gn)$ and its twist $\Gamma_0^i(\gn)$ for each $i=1,...,h$. Setting $\Gamma=\Gamma_0^i(\gn)$ in Definition \ref{abstract partial} and taking suitable modules $V$ to be defined below, we can obtain more concrete partial modular symbols.

For a commutative ring $R$ recall that $V_k(R)$ denotes the space of homogeneous polynomials over $R$ in two variables of degree $k$. Furthermore, for integers $k,\ell\geqslant0$ we define $V_{k,\ell}(R):=V_k(R)\otimes_R V_{\ell}(R)$.

We identify $V_{k,\ell}(R)$ with the space of polynomials that are homogeneous of degree $k$ in two variables $X,Y$ and homogeneous of degree $\ell$ in two further variables $\overline{X}, \overline{Y}$.

\begin{definition}
Let $R$ be a $K$-algebra, we have a left $\Gamma_0^i(\gn)$-action on $V_k(R)$ defined by 
$\gamma\cdot P\binom{X}{Y}= P\binom{dX+bY}{cX+aY}$, for $\gamma=\smallmatrixx{a}{b}{c}{d}$. We obtain a left $\Gamma_0^i(\gn)$-action on $V_{k,\ell}(R)$ given by
\begin{equation*}
\gamma\cdot P\left[\matrixxx{X}{Y},\matrixxx{\overline{X}}{\overline{Y}}\right]=P\left[\matrixxx{dX+bY}{cX+aY},\matrixxx{\overline{d}\overline{X}+\overline{b}\overline{Y}}{\overline{c}\overline{X}+\overline{a}\overline{Y}}\right]. \end{equation*}
This induces a right $\Gamma_0^i(\gn)$-action on the dual space $V_{k,\ell}^*(R)$ by setting 
\begin{equation*}
\mu|_\gamma(P)=\mu(\gamma\cdot P).    
\end{equation*}
\end{definition}

Let $\mathcal{C}=(\mathcal{C}_1,...,\mathcal{C}_h)$ with $\mathcal{C}_i$ a non-empty $\Gamma_0^i(\gn)$-invariant subset of $\mathbb{P}^1(K)$.

\begin{definition}\label{Def5.2}
i) Define the space of \textit{partial Bianchi modular symbols} on $\mathcal{C}_i$ of weight $(k,\ell)$ and level $\Gamma_0^i(\gn)$ to be the space $\mathrm{Symb}_{\Gamma_0^i(\gn),\mathcal{C}_i}(V_{k,\ell}^*(\mathbb{C}))$.\\
ii) Define the space of \textit{partial Bianchi modular symbols} on $\mathcal{C}$ of weight $(k,\ell)$ and level $K_0(\gn)$ to be the space
\begin{equation*}
\mathrm{Symb}_{K_0(\gn),\mathcal{C}}(V_{k,\ell}^*(\mathbb{C})):=\bigoplus_{i=1}^{h} \mathrm{Symb}_{\Gamma_0^i(\gn),\mathcal{C}_i}(V_{k,\ell}^*(\mathbb{C})).
\end{equation*}
\end{definition}

\begin{remark}{\label{fullBianchi}}
When $\mathcal{C}_i=\mathbb{P}^1(K)$ for all $i$, we drop $\mathcal{C}$ from the notation and recover the space $\mathrm{Symb}_{K_0(\gn)}(V_{k,\ell}^*(\mathbb{C}))$ of Bianchi modular symbols in \cite[Def 2.4]{chris2017}.
\end{remark}

\subsection{Partial modular symbols and \texorpdfstring{$C$}{Lg}-cuspidal forms}{\label{partialsymbolsattached}}

Recall the ideal $\gn=(p)\gm$ with $\gm$ coprime to $(p)$. To relate partial modular symbols with $C$-cuspidal Bianchi modular forms we henceforth take $\mathcal{C}=C=(C_1,...,C_h)$ with $C_i=\Gamma_0^i(\gm)\infty\cup\Gamma_0^i(\gm)0$ as in Section \ref{Fourier}, and consider the space of partial Bianchi modular symbols on $C$ of weight $(k,\ell)$ and level $K_0(\gn)$.

We can define Hecke operators on the space of Bianchi modular symbols similarly as were defined on Bianchi modular forms in Section \ref{Fourier}. 

\begin{definition}{\label{definitionoperator}}
Let $\gq$ be a prime ideal, then the Hecke operator $T_\gq$ is defined on the space of Bianchi modular symbols $\mathrm{Symb}_{K_0(\gn)}(V_{k,\ell}^*(\mathbb{C}))$ by
\begin{equation*}
(\phi_1,...,\phi_h)|_{T_\gq}=\left(\phi_{j_1}\bigg|\left[\Gamma_0^{j_1}(\gn)\smallmatrixx{1}{0}{0}{\alpha_1}\Gamma_0^1(\gn)\right],...,\phi_{j_h}\bigg|\left[\Gamma_0^{j_h}(\gn)\smallmatrixx{1}{0}{0}{\alpha_h}\Gamma_0^h(\gn)\right]\right).
\end{equation*}
If $\gq|\gn$ we denote the Hecke operator by $U_\gq$. 
\end{definition}

Analogously with $C$-cuspidal Bianchi modular forms, not all Hecke operators act on $\mathrm{Symb}_{K_0(\gn),C}(V_{k,\ell}^*(\mathbb{C}))$.

\begin{lemma}{\label{actionHeckeoperator}}
The Hecke algebra $\mathcal{H}_{\gn,p}$ acts on $\mathrm{Symb}_{K_0(\gn),C}(V_{k,\ell}^*(\mathbb{C}))$.
\end{lemma}

\begin{proof}
Let $(\phi_1,...,\phi_h)\in \mathrm{Symb}_{K_0(\gn),C}(V_{k,\ell}^*(\mathbb{C}))$, we have to show that for all prime $\gq\nmid\gm$ with $\gq I_i =(\alpha_i)I_{j_i}$ and for each $i=1,...,h$ then
\begin{equation*}
\phi_{j_i}\big|\left[\Gamma_0^{j_i}(\gn)\smallmatrixx{1}{0}{0}{\alpha_i}\Gamma_0^i(\gn)\right]\in \mathrm{Symb}_{\Gamma_0^i(\gn),C_i}(V_{k,\ell}^*(\mathbb{C})).    
\end{equation*}

It suffices to prove $\gamma\cdot s_i\in C_{j_i}$ for all $\gamma=\gamma_{j_i}\smallmatrixx{1}{0}{0}{\alpha_i}\gamma_i \in \left[\Gamma_0^{j_i}(\gn)\smallmatrixx{1}{0}{0}{\alpha_i}\Gamma_0^i(\gn)\right]$ and $s_i\in C_i$.

Using Lemma \ref{l: cuspides}, for all $s_i\in C_{i}$ we have $\gamma_i\cdot s_i=s_i'\in C_i$ by part i), $\smallmatrixx{1}{0}{0}{\alpha_i}\cdot s_i'=s_{j_i}\in C_{j_i}$ by part ii) and $\gamma_{j_i}\cdot s_{j_i}=s_{j_i}'\in C_{j_i}$ again by part i).
\end{proof}

Let $\mathcal{F}$ be a $C$-cuspidal Bianchi eigenform of weight $(k,\ell)$ and level $K_0(\gn)$. For $1\leqslant i \leqslant h$, recall the factor $c_{q,r}^{i}$ from Definition \ref{factor}. Let $\mathcal{X}^{k-q}\mathcal{Y}^q\overline{\mathcal{X}}^{\ell-r}\overline{\mathcal{Y}}^r$ be the element of the dual basis of $V_{k,\ell}^*(\mathbb{C})$ defined by 
\begin{equation*}
\mathcal{X}^{k-q}\mathcal{Y}^q\overline{\mathcal{X}}^{\ell-r}\overline{\mathcal{Y}}^r(X^{k-i}Y^i\overline{X}^{\ell-j}\overline{Y}^j)=
\begin{cases}
1 & : q=i\;\;and\;\;r=j, \\
0 & : otherwise.
\end{cases}
\end{equation*}

\begin{definition}{\label{pbmsattach}}
For each descent $f^i$ of $\mathcal{F}$, we define 
\begin{equation*}
\phi_{f^{i}}\in\mathrm{Hom}(\Delta_{C_i}^0,V_{k,\ell}^*(\mathbb{C}))
\end{equation*}
by setting
\begin{equation}
\phi_{f^{i}}(\{a\}-\{\infty\})= \sum_{q=0}^{k}\sum_{r=0}^\ell c_{q,r}^{i}(a)(\mathcal{Y}-a\mathcal{X})^{k-q}\mathcal{X}^q(\overline{\mathcal{Y}}-\overline{a}\overline{\mathcal{X}})^{\ell-r}\overline{\mathcal{X}}^r.
\end{equation}  
for each $a\in C_i$.
\end{definition}

\begin{proposition}{\label{attach partial}}
We have $\phi_{\mathcal{F}}:=(\phi_{f^{1}},...,\phi_{f^{h}})\in \mathrm{Symb}_{K_0(\gn),C}(V_{k,\ell}^*(\mathbb{C}))$. Moreover, the map $\mathcal{F}\rightarrow\phi_{\mathcal{F}}$ is $\mathcal{H}_{\gn,p}$-equivariant.
\end{proposition}

\begin{proof}
First note that for each $i\in\{1,...,h\}$, the Bianchi modular form $f^i$ is $C_i$-cuspidal, then we can attach to it a vector valued differential 1-form following \cite[\S 2.5]{hida1994critical} (see also $\S 10$ in op.cit.). Just as in \cite[\S 2.4]{chris2017} we can integrate the resulting differential 1-form between cusps in $C_i$ and similarly to \cite[Prop2.9]{chris2017} and \cite[\S5.2]{ghate1999critical}, we obtain $\phi_{f^i}\in \mathrm{Symb}_{\Gamma_0^i(\gn),C_i}(V_{k,\ell}^*(\mathbb{C}))$. The $\mathcal{H}_{\gn,p}$-equivariance is an easy, but lengthy, check.
\end{proof}

\begin{remark}{\label{Remarkevaluarpartial}}
For each $i$ and integers $q,r,$ satisfying $0\leqslant q \leqslant k, 0\leqslant r \leqslant \ell$, we have
\begin{equation*}
\phi_{f^i}(\{a\}-\{\infty\})\left[(X+aY)^qY^{k-q}(\overline{X}+\overline{a}\overline{Y})^{r}\overline{Y}^{\ell-r}\right]=c_{q,r}^i(a), 
\end{equation*}
then we can link $\phi_{\mathcal{F}}$ with the $L$-values of $\mathcal{F}$ using Proposition \ref{T4.6}. 
\end{remark}

\section{\texorpdfstring{$p$}{Lg}-adic \texorpdfstring{$L$}{Lg}-function of \texorpdfstring{$C$}{Lg}-cuspidal Bianchi modular forms}{\label{S4}}

In this section, we define \textit{overconvergent partial Bianchi modular symbols} to link partial symbols with $p$-adic distributions and prove a control theorem. In Section \ref{S4.4} we construct $p$-adic $L$-functions of $C$-cuspidal Bianchi modular forms.

Henceforth, we denote $\roi_{K,p}:=\roi_K\otimes_\Z\Z_p$ to ease notation.
\subsection{Locally analytic distributions}{\label{sectionlocally}}

Suppose $p\roi_K=\prod_{\gp|p}\gp^{e_\gp}$ and define $f_\gp$ to be the residue class degree of $\gp$. Note that $\sum f_\gp e_\gp=2$. Using the embedding $\iota_p:\overline{\Q}\hookrightarrow\overline{\Q}_p$ from Section \ref{notations}, for each prime $\gp|p$, we have $f_\gp e_\gp$ embeddings $K_\gp\hookrightarrow\overline{\Q}_p$, and combining these for each prime, we get an embedding $\sigma:K\otimes\mathbb{Q}_p \hookrightarrow \overline{\Q}_p\times \overline{\Q}_p$ given by $\sigma(a)=(\sigma_1(a),\sigma_2(a))$.

For $r,s\in\R_{>0}$, define the rigid analytic $(r,s)$-neighborhood of $\tensorspace$ in $\C_p^2$ to be 
\begin{equation*}
B(\tensorspace,r,s):=\{(x,y)\in \mathbb{C}_p^2:\exists u\in\tensorspace \; \text{such that}\; |x-\sigma_1(u)|\leqslant r, |y-\sigma_2(u)|\leqslant s\},    
\end{equation*}
where $|\cdot|$ is the norm of $\C_p$ which is normalized as $|p| = p^{-1}$.

Let $L$ be a finite extension of $\mathbb{Q}_p$ containing the image of $\sigma_1$ and $\sigma_2$. We equip $L$ with a valuation $v_p$, normalized so that $v_p(p) = 1$. For $(r,s)$ as above, we write $\mathbb{A}[L,r,s]$ for the $L$-Banach space of rigid analytic functions on $B(\tensorspace,r,s)$ and $\mathbb{D}[L,r,s]$ for its Banach dual (see \cite[\S 5.1]{chris2017}).

Define the space of $L$-valued locally analytic distributions to be the projective limit
\begin{equation*}
\mathcal{D}(L)=\lim_{\longleftarrow}\mathbb{D}[L,r,s]=\bigcap_{r,s}\mathbb{D}[L,r,s]. 
\end{equation*}

We endow $\mathbb{A}[L,r,s]$ with a weight $(k,\ell)$-action of the semigroup
\begin{equation*}
\Sigma_0(p):=\left\{ \smallmatrixx{a}{b}{c}{d}\in M_2(\tensorspace): p|c, a\in\tensorspace^\times, ad-bc\nequal0 \right\}    
\end{equation*}
by setting 
\begin{equation*}
\gamma \cdot_{k,\ell} \zeta(x,y)= (a_1+c_1x)^k(a_2+c_2y)^\ell \zeta\left(\frac{b_1+d_1x}{a_1+c_1x},\frac{b_2+d_2y}{a_2+c_2y}\right);\, \sigma_i(\gamma)=\smallmatrixx{a_i}{b_i}{c_i}{d_i}, 
\end{equation*}
which is well-defined.

These actions are compatible for various $(r,s)$. By duality, they induce actions on $\mathbb{D}[L,r,s]$, and hence on $\mathcal{D}(L)$. We write $\mathbb{D}_{k,\ell}[L,r,s]$ (resp. $\mathcal{D}_{k,\ell}(L)$) for the spaces $\mathbb{D}[L,r,s]$ (resp. $\mathcal{D}(L)$) together with the weight $(k,\ell)$-action of $\Sigma_0(p)$.

\subsection{Overconvergent partial Bianchi modular symbols}{\label{section overconvergent}}

Recall the level $K_0(\gn)$ with $(p)|\gn$ from Section \ref{background} and for each $1\leqslant i \leqslant h$, the $\Gamma_0^i(\gn)$-invariant subset $C_i$ of $\mathbb{P}^1(K)$ from Section \ref{Fourier}. Since the lower left entry of a matrix in $\Gamma_0^i(\gn)$ is in $\gp$ for all $\gp|p$ (because $(p)|\gn$), we have that $\Gamma_0^i(\gn)\subset\Sigma_0(p)$. Then we can equip $\mathcal{D}_{k,\ell}(L)$ with an action of $\Gamma_0^i(\gn)$ and obtain well-defined partial modular symbols with values in $\mathcal{D}_{k,\ell}(L)$.

\begin{definition}
i) Define the space of \textit{overconvergent partial Bianchi modular symbols} on $C_i$ of weight $(k,\ell)$ and level $\Gamma_0^i(\gn)$ with coefficients in $L$ to be the space $\mathrm{Symb}_{\Gamma_0^i(\gn),C_i}(\mathcal{D}_{k,\ell}(L))$.\\
ii) Define the space of \textit{overconvergent partial Bianchi modular symbols} on $C$ of weight $(k,\ell)$ and level $K_0(\gn)$ with coefficients in $L$ to be the space
\begin{equation*}
\mathrm{Symb}_{K_0(\gn),C}(\mathcal{D}_{k,\ell}(L)):=\bigoplus_{i=1}^{h} \mathrm{Symb}_{\Gamma_0^i(\gn),C_i}(\mathcal{D}_{k,\ell}(L)).
\end{equation*}
\end{definition}

Note that the matrices appearing in Definition \ref{definitionoperator} of Hecke operators can be seen inside $\Sigma_0(p)$, then the Hecke algebra $\mathcal{H}_{\gn,p}$ acts on $\mathrm{Symb}_{K_0(\gn),C}(\mathcal{D}_{k,\ell}(L))$.

There is a natural map $\mathcal{D}_{k,\ell}(L)\rightarrow V_{k,\ell}^
*(L)$ given by dualizing the inclusion of $V_{k,\ell}(L)$ into $\mathcal{A}_{k,\ell}(L)$. This induces a $\mathcal{H}_{\gn,p}$-equivariant \textit{specialization map}
\begin{equation*}
\rho: \mathrm{Symb}_{K_0(\gn),C}(\mathcal{D}_{k,\ell}(L)) \longrightarrow \mathrm{Symb}_{K_0(\gn),C}(V_{k,\ell}^*(L)).  
\end{equation*}

Before proving \textit{the partial Bianchi control theorem}, we record the following lemma (analog to \cite[Lem 3.8]{chris2017}) that allows us to work with integral distributions by imitating \cite[Prop 3.9]{chris2017}.

\begin{lemma}{\label{fingen}}
$\Delta_{C_i}^0$ is finitely generated as a $\mathbb{Z}[\Gamma_0^i(\gn)]$-module for $i=1,...,h$.    
\end{lemma}
\begin{proof}
This follows from the fact that $\Gamma_0^i(\gn)$ is a finitely generated group and the set of orbits of the action of $\Gamma_0^i(\gn)$ in $C_i$ is finite.
\end{proof}

\begin{proposition}{\label{partialcontrol}}
(Williams' Bianchi control theorem). For each prime $\gp$ above $p$, let $\lambda_{\gp}\in L^\times$. Suppose that $v_p(\lambda_{\gp})<(\mathrm{min}\{{k,\ell\}}+1)/e_{\gp}$ when $p=\gp^{e_\gp}$ is inert or ramified, or $v_p(\lambda_{\gp})<k+1$ and $v_p(\lambda_{\overline{\gp}})<\ell+1$ when $p$ splits as $\gp\overline{\gp}$ and $\gp,\overline{\gp}$ correspond to the embeddings $\sigma_1,\sigma_2$ from Section \ref{sectionlocally} respectively, then the restriction of the specialization map 
\begin{equation*}
\rho: \mathrm{Symb}_{K_0(\gn),C}(\mathcal{D}_{k,\ell}(L))^{\{U_{\gp}=\lambda_{\gp}:\gp|p\}} \longrightarrow\mathrm{Symb}_{K_0(\gn),C}(V_{k,\ell}^*(L))^{\{U_{\gp}=\lambda_{\gp}:\gp|p\}}    
\end{equation*}
to the simultaneous $\lambda_{\gp}$-eigenspaces of the $U_{\gp}$ operators is an isomorphism.
\end{proposition}

\begin{proof}
This result is proved like its counterpart for Bianchi modular symbols of parallel weight $(k,k)$ in \cite{chris2017}, doing the corresponding adaptations for weight $(k,\ell)$ and cusps. Here we just present a brief idea of how to follow Williams' proof.

First, the control theorem is proved for the space of integral rigid analytic distributions, $\mathbb{D}_{k,\ell}[\roi_L,1,1]$, with the operator $U_p$ using \cite[Cor 4.3]{chris2017} and noting the change of the condition $v_p(\lambda)<k+1$ in \cite[Lem 3.15]{chris2017} to $v_p(\lambda)<\mathrm{min}\{k,\ell\}+1$. Then, we can extend the result for $\mathbb{D}_{k,\ell}[L,1,1]$ and by \cite[Prop 5.2]{chris2017}, we obtain the result for $\mathcal{D}_{k,\ell}(L)$. 

When $p$ is inert, the process described above allows us to prove the theorem. For $p$ ramified as $\gp^2$, we note that $U_\gp=U_{\gp^2}=U_\gp^2$ and use \cite[Lem 6.9]{chris2017} to obtain the result. Finally, when $\gp$ splits, we adapt the methods in \cite[\S6]{chris2017} to lift simultaneous eigensymbols of $U_\gp$ and $U_{\overline{\gp}}$, obtaining the more subtle result on the slope.
\end{proof}

\begin{remark}
Proposition \ref{partialcontrol} can be proved using a cohomological interpretation of partial Bianchi modular symbols as \cite[\S 2.4]{bellaiche2015p} and following \cite{urban}, \cite{hansen}, \cite{canadian}. In this paper we did not use such an interpretation, however, in a forthcoming paper where we use partial Bianchi modular symbols to construct $p$-adic $L$-function in families, cohomology will be key.
\end{remark}

\begin{definition}
Let $\mathcal{F}$ be an eigenform with eigenvalues $\lambda_I$, we say $\mathcal{F}$ has \textit{small slope} if $v_p(\lambda_{\gp})<(\mathrm{min}\{{k,\ell\}}+1)/e_{\gp}$ when $p=\gp^{e_\gp}$ is inert or ramified; or if $v_p(\lambda_{\gp})<k+1$ and $v_p(\lambda_{\overline{\gp}})<\ell+1$ when $p$ splits as $\gp\overline{\gp}$ and $\gp,\overline{\gp}$ correspond to the embeddings $\sigma_1,\sigma_2$ from Section \ref{sectionlocally} respectively. We say $\mathcal{F}$ has \textit{critical slope} if it does not have small slope.
\label{D10.4}
\end{definition}

\subsection{Admissible distributions}{\label{admissible}} 

For each pair $r,s\in \R$, the space $\mathcal{D}_{k,\ell}[L,r,s]$ from Section \ref{sectionlocally} admits an operator norm $||\cdot||_{r,s}$ via
\begin{equation*}
||\mu||_{r,s}= \sup_{0\neq f \in \mathcal{A}_{k,\ell}[L,r,s]} \frac{|\mu(f)|_p}{|f|_{r,s}},
\end{equation*}
where $|\cdot|_p$ is the usual $p$-adic absolute value on $L$ and $|\cdot|_{r,s}$ is the sup norm on $\mathcal{A}_{k,\ell}[L,r,s]$. Note that if $r \leqslant r'$, $s\leqslant s'$, then $||\mu||_{r,s}\geqslant||\mu||_{r',s'}$ for $\mu\in \mathcal{D}_{k,\ell}[L,r',s']$.

These norms give rise to a family of norms on the space of locally analytic functions that allow us to classify locally analytic distributions by growth properties as we vary in this family.

\begin{definition}
Let $\mu\in\mathcal{D}_{k,\ell}(L)$ be a locally analytic distribution.

i) Suppose $p$ is inert or ramified in $K$, we say $\mu$ is $h$-admissible if $||\mu||_{r,r}=O(r^{-h})$ as $r\rightarrow 0^{+}$.

ii) Suppose $p$ splits in $K$, we say $\mu$ is $(h_1,h_2)$-admissible if $||\mu||_{r,s}=O(r^{-h_1})$ uniformly in $s$ as $r\rightarrow 0^{+}$, and $||\mu||_{r,s}=O(r^{-h_2})$ uniformly in $r$ as $s\rightarrow 0^{+}$.

We say that $\mu$ is a measure, if it is bounded, i.e., is $0$-admissible or $(0,0)$-admissible, depending on the splitting behavior of $p$ in $K$.
\end{definition}

\begin{proposition}
Let $\Psi\in \mathrm{Symb}_{\Gamma_0^i(\gn),C_i}(\mathcal{D}_{k,\ell}(L))$, and $r,s\leqslant1$. Defining 
\begin{equation*}
||\Psi||_{r,s}:= \sup_{D\in\Delta_{C_i}^0}||\Psi(D)||_{r,s} 
\end{equation*}
gives a well-defined norm on $ \mathrm{Symb}_{\Gamma_0^i(\gn),C_i}(\mathcal{D}_{k,\ell}(L))$.
\end{proposition}

\begin{proof}
The proof of this proposition is almost identical to the proof of \cite[Def-Prop 5.12]{chris2017}. In fact, first we write $D\in \Delta_{C_i}^0$ as $D=\alpha_1 D_1+\cdots+\alpha_n D_n$ for a finite set of generators $D_j$ and $\alpha_j\in \mathbb{Z}[\Gamma_0^i(\gn)]$ using Lemma \ref{fingen}. Then for every $\gamma\in\Gamma_0^i(\gn)$ and $f\in\mathcal{A}_{k,\ell}[L,r,s]$ we can prove (in the same way of \cite[Lem 5.11]{chris2017}) that exist positive constants $A$ and $A'$ such that
\begin{equation*}
A|\gamma\cdot_{k,\ell} f|_{r,s} \leqslant |f|_{r,s} \leqslant A'|\gamma\cdot_{k,\ell} f|_{r,s}.    
\end{equation*} 
Finally, by the above there exists a constant $B$ such that (without loss of generality)
\begin{equation*}
||\Psi(D)||_{r,s}\leqslant B||\Psi(D_1)||_{r,s}.  
\end{equation*}
In particular, the supremum is finite and hence gives a well-defined norm as required.
\end{proof}

The following proposition is the analog of \cite[Props 5.12, 6.15]{chris2017} for overconvergent partial Bianchi modular symbols.

\begin{proposition}{\label{admissibility}}
Let $\Psi \in \mathrm{Symb}_{\Gamma_0^i(\gn),C_i}(\mathcal{D}_{k,\ell}(L))$.

i) Suppose $p$ is inert in $K$ and $\Psi$ is a $U_p$-eigensymbol with eigenvalue $\lambda$ and slope $h = v_p(\lambda)$. Then, for every $D\in\Delta_{C_i}^0$, the distribution $\Psi(D)$ is $h$-admissible.

ii) Suppose $p$ splits in $K$ as $\gp\overline{\gp}$ and $\Psi$ is simultaneously a $U_{\gp}^n$- and $U_{\overline{\gp}}^n$-eigensymbol for some $n$, with non-zero eigenvalues $\lambda_1^n$ and $\lambda_2^n$ with slopes $h_1=v_p(\lambda_1)$ and $h_2=v_p(\lambda_2)$. Then, for every $D\in\Delta_{C_i}^0$, the distribution $\Psi(D)$ is $(h_1,h_2)$-admissible.
\end{proposition}
\begin{proof}
To prove part i) we follow line by line the proof of \cite[Prop 5.13]{chris2017} and conclude that for any $r$ and positive integer $m$, we have
\begin{equation*}
||\Psi(D)||_{r/p^m,r/p^m}\leqslant |\lambda|_p^{-m}||\Psi||_{r,r}.    
\end{equation*}

To prove part ii) we use the same idea of part i) and conclude that for any $r,s$ using $m$-th powers of the operators $U_\gp^n$ and $U_{\overline{\gp}}^n$ we have 
\begin{equation*}
||\Psi(D)||_{r/p^{mn},s}\leqslant |\lambda_1|_p^{-mn}||\Psi||_{r,s}\;\;\text{and}\;\;||\Psi(D)||_{r,s/p^{mn}}\leqslant |\lambda_2|_p^{-mn}||\Psi||_{r,s},    
\end{equation*}
and combining both, we obtain the result.
\end{proof}

\subsection{Construction of the \texorpdfstring{$p$}{Lg}-adic \texorpdfstring{$L$}{Lg}-function}{\label{S4.4}}

In this section, we recall ray class groups, define the Mellin transform of an overconvergent partial Bianchi modular symbol and finally construct the $p$-adic $L$-function of a $C$-cuspidal Bianchi modular form.

\begin{definition}
Define the ray class group of level $p^\infty$ to be
\begin{equation*}
\mathrm{Cl}_K(p^{\infty})=K^\times\backslash \mathbb{A}_K^\times/\mathbb{C}^\times\prod_{v\nmid p}\roi_v^\times
\end{equation*}
and denote by $\mathfrak{X}(\mathrm{Cl}_K(p^{\infty}))$ the two-dimensional rigid space of $p$-adic characters on $\mathrm{Cl}_K(p^{\infty})$.
\end{definition}

Note that $\mathrm{Cl}_K(p^\infty)$ can be written as $\mathrm{Cl}_K(p^\infty)=\bigcup_{x\in\mathrm{Cl}_K}\mathrm{Cl}_K^x(p^\infty)$, where $\mathrm{Cl}_K^x(p^\infty)$ is the fiber of $x$ under the canonical surjection $\mathrm{Cl}_K(p^\infty)\twoheadrightarrow\mathrm{Cl}_K$ to the class group of $K$, also note that, the choice of $t_i\in \mathbb{A}_K^{f ,\times}$ in Section \ref{background} identifies $\mathrm{Cl}_K^i(p^\infty)$ non-canonically with $\tensorspace^\times/\mathcal{O}_K^\times$.

Let $\Psi = \{\Psi_x\}_{x \in \mathrm{Cl}_K}\in \mathrm{Symb}_{K_0(\gn),C}(\mathcal{D}_{k,\ell}(L))$ be an overconvergent partial Bianchi modular symbol, we define for $i,x\in\mathrm{Cl}_K$ a distribution $\mu_i(\Psi_x)\in\mathcal{D}(\mathrm{Cl}_K^i(p^\infty),L)$ as follows. 

Since $\{0\}-\{\infty\}\in \Delta_{C_i}^0 $ for all $i$, we have a distribution $\Psi_x(\{0\}-\{\infty\})|_{\tensorspace^\times}$ on $\tensorspace^\times$. This restricts to a distribution on $\tensorspace^\times/\mathcal{O}_K^\times$, which gives the distribution $\mu_i(\Psi_x)$ on $\mathrm{Cl}_K^i(p^\infty)$ under the identification above, for more details see \cite[\S 7.4]{chris2017}. 

\begin{definition}{\label{mellin}}
The \textit{Mellin transform} of $\Psi\in \mathrm{Symb}_{K_0(\gn),C}(\mathcal{D}_{k,\ell}(L))$ is the ($L$-valued) locally analytic distribution on $\mathrm{Cl}_K(p^\infty)$ given by 
\begin{equation*}
\mathrm{Mel}(\Psi):=\sum_{i\in\mathrm{Cl}_K}\mu_i(\Psi_i)\in \mathcal{D}(\mathrm{Cl}_K(p^\infty),L).
\end{equation*}
\end{definition}

\begin{remark}
The distribution $\mathrm{Mel}(\Psi)$ is independent of the choice of class group representatives.
\end{remark}

The theory of partial Bianchi modular symbols developed in Sections \ref{S3} and \ref{S4}, allows us to construct the $p$-adic $L$-function of small slope $C$-cuspidal Bianchi modular forms.

\begin{theorem}\label{C-cuspidal p-adic}
Let $\mathcal{F}$ be a small slope $C$-cuspidal Bianchi eigenform of weight $(k,\ell)$ and level $K_0(\gn)$. Let $\iota$ be an isomorphism $\mathbb{C}\xrightarrow{\sim}\overline{\mathbb{Q}}_p$ satisfying $\iota\circ\iota_\infty=\iota_p$. Then there exists a locally analytic distribution $L_p^{\iota}(\mathcal{F},-)$ on $\mathrm{Cl}_K(p^{\infty})$ such that for any Hecke character of $K$ of conductor $\mathfrak{f}=\prod_{\gp|p}\gp^{r_{\gp}}$ and infinity type $(q,r)$ satisfying $0\leqslant q \leqslant k, 0\leqslant r \leqslant \ell$, we have 
\begin{equation}{\label{interpolationprop}}
L_p^{\iota}(\mathcal{F},\psi_{p-\mathrm{fin}})=\left( \prod_{\mathfrak{p}|p} Z_{\mathfrak{p}}(\psi) \right) \left[ \frac{D_Kw_K\tau(\psi)}{(-1)^{\ell+q+r}2\lambda_\mathfrak{f}} \right] \Lambda(\mathcal{F},\psi),
\end{equation}
where $\psi_{p-\mathrm{fin}}$ is the $p$-adic avatar of $\psi$ as in \cite[\S 7.3]{chris2017}, $U_\gf \mathcal{F} = \lambda_\gf \mathcal{F}$ with $U_\gf=\prod_{\gp|p}U_\gp^{r_\gp}$ and
\begin{equation*}
Z_\mathfrak{p}(\psi):=\begin{cases} 1-[\lambda_\gp\psi(\gp)]^{-1} & : \gp\nmid\gf, \\ 1 & :\mbox{else}. \end{cases}
\end{equation*}
The distribution $L_p^{\iota}(\mathcal{F},-)$ is $(h_\mathfrak{p})_{\mathfrak{p}|p}$-admissible, where $h_\mathfrak{p}=v_p(\lambda_\mathfrak{p})$, and therefore is unique.

We call $L_p^{\iota}(\mathcal{F},-)$ the $p$-adic $L$-function of $\mathcal{F}$.
\end{theorem}

\begin{proof}
The small slope $C$-cuspidal Bianchi modular form $\mathcal{F}$ corresponds to a collection of $C_i$-cuspidal forms $f^1,...,f^h$ on $\mathcal{H}_3$. 

We first attach to $\mathcal{F}$ a complex-valued partial Bianchi modular eigensymbol $\phi_\mathcal{F}=(\phi_{f^{1}},...,\phi_{f^{h}})$ using Proposition  \ref{attach partial}. 

By applying $\iota$, we obtain from $\phi_{\mathcal{F}}$ a symbol $\phi^{\iota}_{\mathcal{F}}=(\phi^{\iota}_{f^{1}},...,\phi^{\iota}_{f^{h}})$ with values in $V_{k,\ell}^*(\overline{\mathbb{Q}}_p)$. Since by Lemma \ref{fingen}, $\Delta_{C_i}^0$ is finitely generated as a $\mathbb{Z}[\Gamma_0^i(\gn)]$-module for each $i$, it follows that for all $i$, $\phi^{\iota}_{f^i}$ has values in $V_{k,\ell}^*(L)$ for a sufficiently large $L/\mathbb{Q}_p$ finite (containing the eigenvalues of $\mathcal{F}$), thus $\phi^{\iota}_{\mathcal{F}}\in \mathrm{Symb}_{K_0(\gn),C}(V_{k,\ell}^*(L))$.

Since $\mathcal{F}$ has small slope, we can lift $\phi^{\iota}_{\mathcal{F}}$ to a unique $\Psi^{\iota}_{\mathcal{F}}\in \mathrm{Symb}_{K_0(\gn),C}(\mathcal{D}_{k,\ell}(L))$ using Proposition \ref{partialcontrol}.

Finally, we define the $p$-adic $L$-function of $\mathcal{F}$ as the Mellin transform of $\Psi^{\iota}_{\mathcal{F}}$
\begin{equation*}
L_p^{\iota}(\mathcal{F},-):=\mathrm{Mel}(\Psi^{\iota}_\mathcal{F}).  
\end{equation*}

The interpolation property in (\ref{interpolationprop}) comes from the link between partial Bianchi modular symbols and $\Lambda(\mathcal{F},-)$ in Remark \ref{Remarkevaluarpartial}. Additionally, by Proposition \ref{admissibility} the distribution $L_p^{\iota}(\mathcal{F},-)$ is $(h_\mathfrak{p})_{\mathfrak{p}|p}$-admissible, where $h_\mathfrak{p}=v_p(\lambda_\mathfrak{p})$. Both interpolation and growth properties give the uniqueness of $L_p^{\iota}(\mathcal{F},-)$.
\end{proof}

\begin{remark}{\label{turn-C-cuspi}}
i) Note that in the above theorem we define the $p$-adic avatar $\psi_{p-\mathrm{fin}}$, using our fixed embeddings as in \cite[pag. 190]{hida1993anti} (although here we could use $\iota$ aswell). This fact is important in Theorem \ref{Cor8.3}, since we will not use an isomorphism $\iota$.

ii) In general, non-cuspidal Bianchi modular forms do not need to be $C$-cuspidal. Whilst it is not computed in this article, the author believes that we can consider a linear combination as in \cite[\S 6.1]{bellaiche2015p}, to turn a non-cuspidal Bianchi modular form into $C$-cuspidal and then construct its $p$-adic $L$-function. This relies on the fact that, for example, in the parallel weight case, in the Fourier expansion of a Bianchi modular form $\mathcal{F}=\sum_{n=0}^{2k+2}\mathcal{F}_nX^{2k+2-n}Y^n$, only the constant term of $\mathcal{F}_0$, $\mathcal{F}_{k+1}$ and $\mathcal{F}_{2k+2}$ can be non-trivial (see part ii) in Remark \ref{Remark1.8}) and to achieve $C$-cuspidality we just need to control the constant term of $\mathcal{F}_{k+1}$ at the cusps $0$ and $\infty$. Note that in the case of non-parallel weight $(k,\ell)$, we need to control the constant terms of $\mathcal{F}_{k+1}$ and $\mathcal{F}_{\ell+1}$ both at the cusps 0 and $\infty$.
\end{remark} 

\section{\texorpdfstring{$p$}{Lg}-adic \texorpdfstring{$L$}{Lg}-function of non-cuspidal base change Bianchi modular forms}

Throughout this section, suppose $p$ splits in $K$ as $\gp\overline{\gp}$ with $\gp$ being the prime corresponding to the embedding $\iota_p$ from Section \ref{notations}. 

In this section, we prove algebraicity of the critical $L$-values of a non-cuspidal base change Bianchi modular form $\mathcal{F}$, we attach a complex partial Bianchi modular symbol to $\mathcal{F}$ and view it $p$-adically without using an isomorphism between $\mathbb{C}$ and $\overline{\mathbb{Q}}_p$. Furthermore, we factor the $p$-adic $L$-function of $\mathcal{F}$ as a product of two Katz $p$-adic $L$-functions.

\subsection{Base change and \texorpdfstring{$C$}{Lg}-cuspidality}\label{subsection basechange}

Let $\varphi$ be a Hecke character of $K$ of conductor $\gM$ coprime to $p$ and infinity type $(-k-1,0)$ with $k\geqslant0$. Let $\theta_{\varphi}$ be the theta series associated to $\varphi$, which is a newform of weight $k+2$, level $\Gamma_0(D_KN(\gM))$ and nebentypus $\epsilon_{\theta_{\varphi}}=\chi_K\varphi_{\mathbb{Z}}$ where $\chi_K$ is the quadratic character of $K/\Q$ and $\varphi_{\mathbb{Z}}$ is the Dirichlet character modulo $M=N(\gM)$ given by $\varphi_{\mathbb{Z}}(a)=\varphi(a\mathcal{O}_K)a^{-k-1}$, for an integer $a$ coprime to $M$.

Let $\pi$ be the automorphic representation of $\mathrm{GL}_2(\mathbb{A}_{\mathbb{Q}})$ generated by $\theta_{\varphi}$ and let BC($\pi$) be the base change of $\pi$ to $\mathrm{GL}_2(\mathbb{A}_K)$ (see \cite{langlands1980base}). The \textit{base change} of $\theta_{\varphi}$ to $K$ is the normalized new vector $\theta_{\varphi/K}$ in BC($\pi$) which is a non-cuspidal Bianchi modular form of weight $(k, k)$, level $K_0(\gm)$ with $M\mathcal{O}_K|\gm|D_KM\mathcal{O}_K$ (see \cite[\S 2.3]{friedberg1983imaginary}) and nebentypus $\epsilon_{\theta_{\varphi/K}}=\epsilon_{\theta_{\varphi}}\circ N$.

\begin{remark}\label{remark6.3}
The Hecke eigenvalues of $\theta_{\varphi/K}$ can be described in terms of the eigenvalues $a_q$ of $\theta_{\varphi}$, which are described in turn in terms of $\varphi$: for every prime $\gq$ of $K$ above $q$,
$$
a_{\gq}=\begin{cases}
a_q=\varphi(\gq)+\varphi(\overline{\gq}), & \text{if $q= \gq\overline{\gq}$}\\
a_q=\varphi(\gq), & \text{if $q$ ramifies}\\
a_q^2-2\chi_K(q)\varphi_{\mathbb{Z}}(q)q^{k+1}=2\varphi(\gq) & \text{if $q$ is inert.}
\end{cases}
$$
\end{remark}

Note that the Bianchi modular form $\theta_{\varphi/K}$ has level $\gm$ coprime to $p$, however, to construct its $p$-adic $L$-function we need $p$ in the level, then, we consider its $p$-stabilizations. We are interested in the $p$-stabilization satisfying the small slope condition of Definition \ref{D10.4}, i.e., the $p$-adic valuation of the eigenvalues of $U_\gp$ and $U_{\overline{\gp}}$ are both less than $k+1$. 

The $p$-stabilizations of $\theta_{\varphi/K}$ can be explicitly described from the $p$-stabilizations of $\theta_{\varphi}$ by considering its Hecke polynomial at $p$ given by $x^2-a_px+ \epsilon_{\theta_{\varphi}}(p)p^{k+1}$. Since $p=\gp\overline{\gp}$, we have that
$a_p=\varphi(\gp)+\varphi(\overline{\gp})$ and $\epsilon_{\theta_{\varphi}}(p)=\chi_K(p)\varphi_{\mathbb{Z}}(p)=\varphi(p\mathcal{O}_K)/p^{k+1}=\varphi(\gp\overline{\gp})/p^{k+1}$,
thus the roots of the Hecke polynomial of $\theta_{\varphi}$ at $p$ are $\alpha=\varphi(\overline{\gp})$ and $\beta=\varphi(\gp)$. 

Consider the Hecke polynomial of $\theta_{\varphi/K}$ at $\gq|p$ given by $x^2-a_{\gq}x+ \epsilon_{\theta_{\varphi/K}}(\gq)N(\gq)^{k+1}$. Since
$a_\gq=a_p$ by Remark \ref{remark6.3}, $N(\gq)=p$ and $\epsilon_{\theta_{\varphi/K}}(\gq)=\chi_K(N(\gq))\varphi_{\mathbb{Z}}(N(\gq))=\chi_K(p)\varphi_{\mathbb{Z}}(p)=\epsilon_{\theta_{\varphi}}(p)$; then the Hecke polynomial of $\theta_{\varphi/K}$ at each $\gq|p$ and the Hecke polynomial of $\theta_{\varphi}$ at $p$ are equal. 

Therefore, we can take the roots of the Hecke polynomial of $\theta_{\varphi/K}$ at $\gp$ and at $\overline{\gp}$, to be $\alpha_\gp=\alpha_{\overline{\gp}}=\alpha=\varphi(\overline{\gp})$ and $\beta_\gp=\beta_{\overline{\gp}}=\beta=\varphi(\gp)$, respectively.

If $\theta_{\varphi}^{\alpha}$ (resp. $\theta_{\varphi}^{\beta}$) is the $p$-stabilization of $\theta_{\varphi}$ corresponding to $\alpha$ (resp. $\beta$), we define its base change to $K$ to be the $p$-stabilization $\theta_{\varphi/K}^{\alpha\alpha}$ (resp. $\theta_{\varphi/K}^{\beta\beta}$) of $\theta_{\varphi/K}$ corresponding to $\alpha_\gp$, $\alpha_{\overline{\gp}}$ (resp. $\beta_\gp$, $\beta_{\overline{\gp}}$).

\begin{remark}{\label{basechangepstabilization}}
Note that $v_p(\alpha_\gp)=v_p(\alpha_{\overline{\gp}})=v_p(\varphi(\overline{\gp}))=0$ and $v_p(\beta_\gp)=v_p(\beta_{\overline{\gp}})=v_p(\varphi(\gp))=k+1 $. Since we are interested in the small slope $p$-stabilization of $\theta_{\varphi/K}$, henceforth we will work with $\theta_{\varphi/K}^p:=\theta_{\varphi/K}^{\alpha\alpha}$.
\end{remark}

Recall the level of $\theta_{\varphi/K}^p$ given by $K_0(\gn)$ with $\gn=(p)\gm$ and consider $C=(C_1,...,C_h)$ with $C_i=\Gamma_0^i(\gm)\infty\cup\Gamma_0^i(\gm)0$ as in Section \ref{Fourier}. 

The following proposition is the analog of \cite[Def 6.2]{bellaiche2015p}, however, in our case we do not need to produce a linear combination, since we can take advantage of the nice vanishing properties of $\theta_{\varphi/K}$.

\begin{proposition}{\label{basechangequasi}}
The Bianchi modular form $\theta_{\varphi/K}^p$ is $C$-cuspidal.
\end{proposition}

\begin{proof}
Suppose that $K$ has class number $1$. By part ii) of Remark \ref{Remark1.8}, the constant term of the Fourier expansion at $\infty$ of $\theta_{\varphi/K}=(f_0,...,f_n,...,f_{2k+2})$ is trivial for $n\notin\{0,k+1,2k+2\}$ and by \cite[Thm 3.1]{friedberg1983imaginary}, we also have that the constant Fourier coefficient of $f_{k+1}$ is trivial and therefore $\theta_{\varphi/K}$ quasi-vanishes at $\infty$. Likewise, since $\theta_{\varphi/K}$ is an eigenform for the Hecke operators, then it is also an eigenform for the Fricke involution $W_\gm$, with eigenvalue $\pm1$ (see \cite[\S2]{cremona1994}, \cite[\S 2.4]{palacios}). Note that $W_\gm$ sends the cusp $\infty$ to the cusp $0$, then $\theta_{\varphi/K}$ quasi-vanishes at $0$, therefore $\theta_{\varphi/K}$ quasi-vanishes at $C$. To conclude that $\theta_{\varphi/K}^p$ is $C$-cuspidal, we notice that $C$ is stable under multiplication by the matrix $\smallmatrixx{\varpi_\gp}{0}{0}{1}$ (resp. $\smallmatrixx{\overline{\varpi}_\gp}{0}{0}{1}$) used in the $\gp$-stabilization (resp. $\overline{\gp}$-stabilization) of $\theta_{\varphi/K}$.

When $K$ has a higher class number, we proceed similarly as above and obtain that $\theta_{\varphi/K}^i$ quasi-vanishes at $C_i$ for each $i$. Then we note that the set
of cusps $C_i$
is stable under multiplication by the matrix ${\smallmatrixx{\alpha_i}{0}{0}{1}}$ with $\gp I_i =(\alpha_i)I_{j_i}$
(resp. ${\smallmatrixx{\beta_i}{0}{0}{1}}$ with $\overline{\gp} I_i =(\beta_i)I_{j_i}$) appearing in the $i$-component of the $\gp$-stabilization (resp. $\overline{\gp}$-stabilization)
of $\theta_{\varphi/K}$. 
\end{proof}

\subsection{\texorpdfstring{$L$}{Lg}-function}{\label{basechangeLfunction}}

The fact that $\theta_{\varphi/K}$ is constructed from a Hecke character, allows us to relate its complex $L$-function with Hecke $L$-functions and prove algebraicity of critical $L$-values of $\theta_{\varphi/K}^p$. 

\begin{lemma}\label{Prop 9.1}
Let $\psi$ be a Hecke character of $K$ and let $\psi^c(\gq):=\psi(\overline{\gq})$ where $\gq$ is an ideal of $K$ and $\overline{\gq}$ is its conjugate ideal. Then 
\begin{equation*}
L(\theta_{\varphi/K},\psi,s)=L(\varphi^c\psi,s)L(\varphi^c\psi^c\lambda_K,s)=L(\varphi^c\psi\lambda_K,s)L(\varphi^c\psi^c,s)    
\end{equation*}
where $\lambda_K=\chi_{K}\circ N$.
\end{lemma}

\begin{proof}
This follows by comparing the Euler factors on both sides, using Remark \ref{remark6.3}.
\end{proof}

\begin{remark}
There are 6 more ways to factor $L(\theta_{\varphi/K},\psi,s)$ as product of two Hecke $L$-functions, this comes from the fact that for a Hecke character $\nu$ we have $L(\nu,s)=L(\nu^c,s)$. Similar factorizations in the $p$-adic setting do not necessarily hold, because if $\nu$ has infinity type $(q,r)$ the involution $\nu \rightarrow \nu^c$
corresponds to the map $(q,r)\rightarrow(r,q)$ on weight space and therefore does not preserve the lower right quadrant of weights of Hecke characters that lie in the range
of classical interpolation of the Katz $p$-adic $L$-functions (see \cite[Figure 1]{bertolini2012p} and Theorem \ref{Katztheorem}).
\end{remark}

We now prove algebraicity of the critical $L$-values of $\theta_{\varphi/K}$. For this, we first follow \cite[\S 2.3]{berger} to recall algebraicity of the critical $L$-value $L(\chi,0)$, of a Hecke character $\chi$ of infinity type $(a,b)$ with $a, b \in \Z$ satisfying $a> 0\geqslant b$.

Let $E$ be an elliptic curve with complex multiplication by $\roi_K$ and $\omega$ its Néron differential. Suppose also that $E$ has good reduction at the place above $p$ and $\overline{\omega}$ is a non-vanishing invariant differential on the reduced curve $\overline{E}$, then Damerell showed that $\pi^{-b}\Omega_\infty^{-a+b}L(\chi,0)\in\overline{\Q}$, where $\Omega_\infty\in \C$ is the complex period of $\omega$.

\begin{definition}{\label{periodberger}}
Define the complex period $\Omega_{\theta_{\varphi/K}}\in\mathbb{C}^\times$ attached to $\theta_{\varphi/K}$ to be 
\begin{equation*}
\Omega_{\theta_{\varphi/K}}=\left(\frac{\Omega_\infty}{\pi }\right)^{2k+2}    
\end{equation*}
\end{definition}

\begin{proposition}{\label{period base change}}
For all Hecke character $\psi$ of $K$ with infinity type $(q,r)$ with $q,r\in\mathbb{Z}$ satisfying $0\leqslant q,r \leqslant k$, we have $\Lambda(\theta_{\varphi/K},\psi)/\Omega_{\theta_{\varphi/K}}\in \overline{\mathbb{Q}}$.
\end{proposition}

\begin{proof}
By Lemma \ref{Prop 9.1} we have
\begin{equation*}
L(\theta_{\varphi/K},\psi,1)=L(\varphi^c\psi,1)L(\varphi^c\psi^c\lambda_K,1)=L(\varphi^c\psi|\cdot|_{\mathbb{A}_K},0)L(\varphi^c\psi^c\lambda_K|\cdot|_{\mathbb{A}_K},0).    
\end{equation*} 

Note that $\varphi^c\psi|\cdot|_{\mathbb{A}_K}$ has infinity type $(q+1,r-k)$ with $q+1>0\geqslant r-k$ and $\varphi^c\psi^c\lambda_K|\cdot|_{\mathbb{A}_K}$ has infinity type $(r+1,q-k)$ with $r+1>0\geqslant q-k$, therefore
\begin{equation*}
\frac{L(\varphi^c\psi|\cdot|_{\mathbb{A}_K},0)}{\pi ^{r-k}\Omega_\infty^{q+1-r+k}}\in\overline{\mathbb{Q}}\;\;\;\;\textrm{and}\;\;\;\;
\frac{L(\varphi^c\psi^c\lambda_K|\cdot|_{\mathbb{A}_K},0)}{\pi^{q-k}\Omega_\infty^{r+1-q+k}}\in\overline{\mathbb{Q}}.
\end{equation*}
The result follows by multiplying the two numbers above and noting that 
\begin{equation*}
\frac{L(\theta_{\varphi/K},\psi,1)}{\pi^{q+r-2k}\Omega_\infty^{2k+2}}=\frac{(2i)^{q+r+2}}{q!r!}\frac{\Lambda(\theta_{\varphi/K},\psi)}{\Omega_{\theta_{\varphi/K}}}.
\end{equation*}
\end{proof}
Suppose $\theta_{\varphi/K}$ does not have level at $p$ and consider its $p$-stabilization $\theta_{\varphi/K}^p$ as in Remark \ref{basechangepstabilization}.

\begin{corollary}{\label{lemma6.9}}
For all  Hecke character $\psi$ of $K$ with infinity type $(q,r)$ satisfying $0\leqslant q,r \leqslant k$, we have $\Lambda(\theta_{\varphi/K}^p,\psi)/\Omega_{\theta_{\varphi/K}}\in \overline{\mathbb{Q}}$.
\end{corollary}

\begin{proof} 
The $L$-functions of $\theta_{\varphi/K}^p$ and $\theta_{\varphi/K}$ are related (see \cite[\S3.3]{palacios} for Bianchi modular forms with trivial nebentypus and $K$ with class number 1) by  
\begin{equation*}
\Lambda(\theta_{\varphi/K}^p,\psi)=\prod_{\gq|p}\left(1-\frac{\varphi(\gp)\psi(\gq)}{N(\gq)}\right)\Lambda(\theta_{\varphi/K},\psi).   
\end{equation*}
By Proposition \ref{period base change} we obtain the result.
\end{proof}

\subsection{\texorpdfstring{$p$}{Lg}-adic \texorpdfstring{$L$}{Lg}-function}{\label{s5.4}} 

Let $f_1,..., f_h$ be the collection of descents of $\theta_{\varphi/K}^p$ to $\mathcal{H}_3$. By Proposition \ref{basechangequasi} each $f^i$ is $C_i$-cuspidal, then  we can consider the coefficients $c_{q,r}^i(\cdot)$ from Definition \ref{factor}, and using Proposition \ref{attach partial} we can attach to $\theta_{\varphi/K}^p$ a partial Bianchi modular symbol $\phi_{\theta_{\varphi/K}^p}=(\phi_{f^1},...,\phi_{f^h})$ where
\begin{equation*}
\phi_{f^i}(\{a\}-\{\infty\})= \sum_{q,r=0}^{k}c_{q,r}^i(a)(\mathcal{Y}-a\mathcal{X})^{k-q}\mathcal{X}^q(\overline{\mathcal{Y}}-\overline{a}\overline{\mathcal{X}})^{k-r}\overline{\mathcal{X}}^r. 
\end{equation*}

\begin{proposition}{\label{Prop6.7}}
Let $\Omega_{\theta_{\varphi/K}}$ be the period in Proposition \ref{period base change}, then the partial Bianchi modular symbol $\phi_{\theta_{\varphi/K}^p}^\mathrm{alg}:=\phi_{\theta_{\varphi/K}^p}/\Omega_{\theta_{\varphi/K}}$ takes values in $V_{k,k}^*(E)$ for some number field $E$.
\end{proposition}
\begin{proof}
Let $\psi$ be a Hecke character of $K$ of infinity type $(q,r)$ satisfying $0\leqslant q,r \leqslant k$, then by Proposition \ref{T4.6} we have
\begin{equation*}
\Lambda(\theta_{\varphi/K}^p,\psi)=\left[\frac{(-1)^{k+q+r}2}{D_Kw_K\tau(\psi)}\right]\sum_{i=1}^{h}\left[\psi(t_i)\sum_{\substack{[a]\in\gf^{-1}/\mathcal{O}_K \\ ((a)\gf,\gf)=1,\;a\in C_i}}\psi_\gf(a)c_{q,r}^i(a)\right].
\end{equation*}
Dividing both sides by $\Omega_{\theta_{\varphi/K}}$ we obtain
\begin{equation}{\label{e.6.2}}
\left[\frac{(-1)^{k+q+r}2}{D_Kw_K\tau(\psi)}\right]^{-1}\frac{\Lambda(\theta_{\varphi/K}^p,\psi)}{\Omega_{\theta_{\varphi/K}}}=\sum_{i=1}^{h}\left[\psi(t_i)\sum_{\substack{[a]\in\gf^{-1}/\mathcal{O}_K \\ ((a)\gf,\gf)=1,\;a\in C_i}}\psi_\gf(a)\frac{c_{q,r}^i(a)}{\Omega_{\theta_{\varphi/K}}}\right].
\end{equation}

By Corollary \ref{lemma6.9} we have $\Lambda(\theta_{\varphi/K}^p,\psi)/\Omega_{\theta_{\varphi/K}}\in\overline{\mathbb{Q}}$, then the left-hand side of (\ref{e.6.2}) is an algebraic number and consequently by linear independence of characters we have that $c_{q,r}^i(a)/\Omega_{\theta_{\varphi/K}}\in\overline{\mathbb{Q}}$ for each $i$. 

Consider the polynomial $(X+aY)^qY^{k-q}(\overline{X}+\overline{a}\overline{Y})^{r}\overline{Y}^{k-r}$, we have that
\begin{equation*}
\phi_{f^i}(\{a\}-\{\infty\})\left[(X+aY)^qY^{k-q}(\overline{X}+\overline{a}\overline{Y})^{r}\overline{Y}^{k-r}\right]=c_{q,r}^i(a). 
\end{equation*}
Then, defining $\phi_{f^i}^\mathrm{alg}:=\phi_{f^i}/\Omega_{\theta_{\varphi/K}}$ we have 
\begin{equation}{\label{algebra}}
\phi_{f^i}^\mathrm{alg}(\{a\}-\{\infty\})\left[(X+aY)^qY^{k-q}(\overline{X}+\overline{a}\overline{Y})^{r}\overline{Y}^{k-r}\right]=c_{q,r}^i(a)/\Omega_{\theta_{\varphi/K}}\in\overline{\mathbb{Q}}.    
\end{equation}

Now, using (\ref{algebra}), we will show $\phi_{f^i}^\mathrm{alg}\in \mathrm{Symb}_{\Gamma_0^i(\gn),C_i}(V_{k,k}^*(\overline{\Q}))$. For this, consider some divisor $D\in\Delta_{C_i}^0$ and a polynomial $P\left[\binom{X}{Y}\binom{\overline{X}}{\overline{Y}}\right]\in V_{k,k}(\overline{\Q})$, we want to show $[\phi_{f^i}^\mathrm{alg}(D)](P)\in\overline{\Q}$. 

By Lemma \ref{fingen}, $\Delta_{C_i}^0$ is finitely generated as a $\mathbb{Z}[\Gamma_0^i(\gn)]$-module by divisors $\{a\}-\{\infty\}$, with $a\in C_i$, then, it is enough to prove $[\phi_{f^i}^\mathrm{alg}(\{a\}-\{\infty\})](P)\in\overline{\Q}$.

Let $P\left[\binom{X}{Y}\binom{\overline{X}}{\overline{Y}}\right]=\sum_{b,d=0}^{k}t_{b,d}X^bY^{k-b}\overline{X}^d\overline{Y}^{k-d}$ with $t_{b,d}\in\overline{\Q}$. For $0\leqslant b,d \leqslant k$, we can write each $X^bY^{k-b}\overline{X}^d\overline{Y}^{k-d}$ as a homogeneous polynomial $Q_{b,d}\left[\binom{X+aY}{Y}\binom{\overline{X}+\overline{a}\overline{Y}}{\overline{Y}}\right]$ after replacing $X$ by $(X+aY)-aY$ and $\overline{X}$ by $(\overline{X}+\overline{a}\overline{Y})-\overline{a}\overline{Y}$. Applying (\ref{algebra}) on each $Q_{b,d}$, we obtain $[\phi_{f^i}^\mathrm{alg}(\{a\}-\{\infty\})](Q_{b,d})\in\overline{\Q}$ and then $[\phi_{f^i}^\mathrm{alg}(\{a\}-\{\infty\})](P)\in\overline{\Q}$.

Finally, using again that $\Delta_{C_i}^0$ is finitely generated as a $\mathbb{Z}[\Gamma_0^i(\gn)]$-module, we have for all $i$ that $\phi_{f^i}^\mathrm{alg}\in\mathrm{Symb}_{\Gamma_0^i(\gn),C_i}(V_{k,k}^*(E))$ for some sufficiently large number field $E$. Therefore, $\phi_{\theta_{\varphi/K}^p}^\mathrm{alg}\in\mathrm{Symb}_{K_0(\gn),C}(V_{k,k}^*(E))$.
\end{proof}

The previous proposition and the work done in Section \ref{S4} allow us to obtain the $p$-adic $L$-function of $\theta_{\varphi/K}^p$ without using an isomorphism between $\C$ and $\overline{\Q}_p$. 

\begin{theorem}{\label{Cor8.3}} 
Let $\varphi$ be a Hecke character of $K$ with conductor coprime to $p$ and infinity type $(-k-1,0)$ for $k\geqslant0$, and denote by $\theta_{\varphi}^p$ the ordinary $p$-stabilization of the CM form $\theta_{\varphi}$ induced by $\varphi$. Let $\theta_{\varphi/K}^p$ be the base change to $K$ of $\theta_{\varphi}^p$ and $\Omega_{\theta_{\varphi/K}}$ be the complex period of Definition \ref{periodberger}. Then there exists a unique measure $L_p(\theta_{\varphi/K}^p,-)$ on $\mathrm{Cl}_K(p^{\infty})$ such that for any Hecke character $\psi$ of $K$ of conductor $\gf=\gp^t\overline{\gp}^s$ and infinity type $(q,r)$ satisfying $0\leqslant q,r \leqslant k$, we have 
\begin{equation}\label{equation6.3}
L_p(\theta_{\varphi/K}^p,\psi_{p-\mathrm{fin}})=\left[\prod_{\gq|p}\left(1-\frac{1}{\varphi(\overline{\gp})\psi(\gq)}\right)\right]\left[ \frac{D_Kw_K\tau(\psi)}{(-1)^{k+q+r}2 \varphi(\overline{\gp})^{t+s}\Omega_{\theta_{\varphi/K}}} \right] \Lambda(\theta_{\varphi/K}^p,\psi).
\end{equation}
\end{theorem}

\begin{proof}
By Proposition \ref{attach partial} we can attach to $\theta_{\varphi/K}^p$ a complex-valued partial Bianchi modular symbol $\phi_{\theta_{\varphi/K}^p}$. By Proposition \ref{Prop6.7}, the partial Bianchi modular symbol $\phi_{\theta_{\varphi/K}^p}^\mathrm{alg}=\phi_{\theta_{\varphi/K}^p}/\Omega_{\theta_{\varphi/K}}$ has values in $V_{k,k}^*(L)$ for a sufficiently large $p$-adic field $L$. Since $\theta_{\varphi/K}^p$ has small slope, we can lift $\phi_{\theta_{\varphi/K}^p}^\mathrm{alg}$ to its corresponding unique overconvergent partial Bianchi eigensymbol $\Psi_{\theta_{\varphi/K}^p}$ using Proposition \ref{partialcontrol}. Taking the Mellin transform of $\Psi_{\theta_{\varphi/K}^p}$ we obtain a locally analytic distribution $L_p(\theta_{\varphi/K}^p,-)=\mathrm{Mel}(\Psi_{\theta_{\varphi/K}^p})$ on $\mathrm{Cl}_K(p^{\infty})$ that is $(h_\gp,h_{\overline{\gp}})$-admissible, where $h_\mathfrak{p}=v_p(\lambda_\mathfrak{p})=0$ and $h_{\overline{\gp}}=v_p(\lambda_{\overline{\gp}})=0$, thus is bounded, i.e., it is a measure.

Using the connection between Bianchi modular symbols and $L$-values, for any Hecke character $\psi$ of $K$ of conductor $\gf=\gp^t\overline{\gp}^s$ and infinity type $(q,r)$ satisfying $0\leqslant q,r \leqslant k$, we obtain a similar interpolation to Theorem \ref{C-cuspidal p-adic} given by
\begin{equation*}
L_p(\theta_{\varphi/K}^p,\psi_{p-\mathrm{fin}})=\left[ \prod_{\gp|p} \left(1-\frac{1}{\lambda_\gp\psi(\gp)}\right) \right]\left[ \frac{D_Kw_K\tau(\psi)}{(-1)^{k+q+r}2\lambda_\gf\Omega_{\theta_{\varphi/K}}} \right] \Lambda(\theta_{\varphi/K}^p,\psi).
\end{equation*}
Since $\lambda_\gp=\lambda_{\overline{\gp}}=\varphi(\overline{\gp})$ and $\lambda_\gf=\lambda_\gp^t\lambda_{\overline{\gp}}^s=\varphi(\overline{{\gp}})^{t+s}$, we obtain the interpolation desired, which gives us the uniqueness of $L_p(\theta_{\varphi/K}^p,-)$.
\end{proof}

The measure $L_p^\iota(\theta_{\varphi/K}^p,-)$ from Theorem \ref{C-cuspidal p-adic} and $L_p(\theta_{\varphi/K}^p,-)$ can be related using the factor $\iota(\Omega_{\theta_{\varphi/K}})$.

\begin{proposition}{\label{equality of measures}}
We have the following equality of measures on $\mathrm{Cl}_K(p^\infty)$ 
\begin{equation*}
L_p(\theta_{\varphi/K}^p,-)=\frac{L_p^\iota(\theta_{\varphi/K}^p,-)}{\iota(\Omega_{\theta_{\varphi/K}})}.
\end{equation*}
\end{proposition}

\begin{proof}
Let $\psi$ be a Hecke character of $K$ of conductor $\gf|(p^\infty)$ and infinity type $(q,r)$ satisfying $0\leqslant q,r \leqslant k$, we write 
$$B(\theta_{\varphi/K}^p,\psi)=\left[ \prod_{\gp|p} \left(1-\frac{1}{\lambda_\gp\psi(\gp)}\right) \right]\left[ \frac{D_Kw_K\tau(\psi)}{(-1)^{k+q+r}2\lambda_\gf} \right],$$ for the factor appearing in Theorem \ref{C-cuspidal p-adic}.

By Theorem \ref{C-cuspidal p-adic} we have the interpolation property
\begin{align*}
L_p^\iota(\theta_{\varphi/K}^p,\psi_{p-\mathrm{fin}})&=\iota\left(B(\theta_{\varphi/K}^p,\psi)\Lambda(\theta_{\varphi/K}^p,\psi)\right)\\
&=\iota\left(B(\theta_{\varphi/K}^p,\psi)\frac{\Lambda(\theta_{\varphi/K}^p,\psi)}{\Omega_{\theta_{\varphi/K}}}\right)\iota(\Omega_{\theta_{\varphi/K}}).
\end{align*}

Since $B(\theta_{\varphi/K}^p,\psi)$ is rational and $\frac{\Lambda(\theta_{\varphi/K}^p,\psi)}{\Omega_{\theta_{\varphi/K}}}$ as well by Corollary \ref{lemma6.9}, then
$$L_p^\iota(\theta_{\varphi/K}^p,\psi_{p-\mathrm{fin}})=\iota_p\left(\iota_\infty^{-1}\left(B(\theta_{\varphi/K}^p,\psi)\frac{\Lambda(\theta_{\varphi/K}^p,\psi)}{\Omega_{\theta_{\varphi/K}}}\right)\right)\iota(\Omega_{\theta_{\varphi/K}}),$$
and note that the first factor is the one appearing in the interpolation property for $L_p(\theta_{\varphi/K}^p,\psi_{p-\mathrm{fin}})$.

We deduce that 
\begin{equation*}
L_p(\theta_{\varphi/K}^p,\psi_{p-\mathrm{fin}})=\frac{L_p^\iota(\theta_{\varphi/K}^p,\psi_{p-\mathrm{fin}})}{\iota(\Omega_{\theta_{\varphi/K}})},
\end{equation*}
for infinitely many characters $\psi_{p-\mathrm{fin}}$. Since both sides of the equation are bounded functions, the result follows because a non-zero bounded analytic function on an open ball has at most finitely many zeros.
\end{proof}

\subsection{Relation with Katz \texorpdfstring{$p$}{Lg}-adic \texorpdfstring{$L$}{Lg}-functions}{\label{s5.5}}

To relate the $p$-adic $L$-function of $\theta_{\varphi/K}^p$ with Katz $p$-adic $L$-functions, we first introduce some notation and state the interpolation property of Katz $p$-adic $L$-functions, for more details the reader can see \cite{bertolini2012p} and \cite{hida1993anti}.

Let $\psi$ be a Hecke character of $K$ of suitable infinity type and conductor $\mathfrak{g}\gf$ with $\mathfrak{g}$ coprime to $p$ and $\gf|p^\infty$, then Katz constructed in \cite{katz1978p} the $p$-adic $L$-function of $\psi$ when $\mathfrak{g}$ is trivial. Later, Hida and Tilouine in \cite{hida1993anti} extended Katz' construction for non-trivial $\mathfrak{g}$. 

Recall the Gauss sum of $\psi$ from Section \ref{L-C-cuspi}, we now define the \textit{local} Gauss sum of $\psi$ at prime ideals $\gq$ dividing the conductor of $\psi$ by
\begin{equation*}
\tau_{\gq}(\psi):=\psi(\varpi_{\gq}^{-t})\sum_{u\in(\mathcal{O}_{\gq}/{\gq}^t)^\times}\psi_{\gq}(u)\textbf{e}_K(ud_{\gq}^{-1}\varpi_{\gq}^{-t}) 
\end{equation*}
where $\textbf{e}_K$ is the character in Section \ref{Fourier}, $\varpi_\gq$ is a prime element in $\mathcal{O}_\gq$, $t=t(\gq)$ is the exponent of $\gq$ in the conductor of $\psi$ and $d_\gq$ is the $\gq$ component of the idele $d$ associated to the different ideal of $K$. Outside the conductor of $\psi$, we simply put $\tau_{\gq}(\psi)=1$.

Recall the elliptic curve $E$ with complex multiplication by $\mathcal{O}_K$ of Section \ref{basechangeLfunction}. In Proposition \ref{period base change}, we obtained algebraicity of critical $L$-values of $\psi$ using the complex period $\Omega_\infty$. Analogously, there exists a $p$-adic period $\Omega_p\in\hat{\mathcal{O}}_{\mathrm{ur}}^{\times}$, where $\hat{\mathcal{O}}_{\mathrm{ur}}$ is the ring of integers of the maximal unramified extension of $\Q_p$, that gives us algebraicity of the $p$-adic $L$-values of $\psi_{p-\mathrm{fin}}$. Note that each choice of $\Omega_\infty$ and $\Omega_p$ depends on some algebraic numbers, however the ratio of them does no depend on the choices (see \cite[\S 0]{hida1993anti}).

\begin{theorem}{\label{Katztheorem}}(Katz, Hida-Tilouine)
Let $\mathfrak{g}$ be an ideal of $K$ coprime to $p$. There exists a unique measure $L_p(-)$ on the ray class group $\mathrm{Cl}_K(\mathfrak{g} p^\infty)$ whose value on the $p$-adic avatar $\psi_{p-\text{fin}}$ of a Hecke character $\psi$ of $K$ of infinity type $(a,b)$ with $a>0\geqslant b$ and conductor $\mathfrak{g}\gf$ with $\gf=\gp^t\overline{\gp}^s$ is given by: 
\begin{equation*}
\frac{L_p(\psi_{p-\text{fin}})}{\Omega_p^{a-b}}=\frac{\Gamma(a)\sqrt{D_K}^{b}\tau_{\gp}(\psi)}{2(-1)^{a+b}w_K^{-1}2^{b}N(\gp)^{t}}(1-\psi(\overline{\gp}))\left(1-[\psi(\gp)N(\gp)^t]^{-1}\right) \frac{L(\psi,0)}{\pi^b\Omega_\infty^{a-b}}. \end{equation*}
\end{theorem}

In order to link our $p$-adic $L$-function $L_p(\theta_{\varphi/K}^p,-)$ with Katz $p$-adic $L$-functions, we use the factorization of $L(\theta_{\varphi/K},\psi,s)$ as the product of two Hecke $L$-functions in Section \ref{basechangeLfunction}. For this, we first rewrite the interpolation property of $L_p(\theta_{\varphi/K}^p,-)$  by considering the relation between $\Lambda(\theta_{\varphi/K}^p,-)$ and $\Lambda(\theta_{\varphi/K},-)$ in the proof of Corollary \ref{lemma6.9}. Then, by Theorem \ref{Cor8.3}, for any Hecke character $\psi$ of $K$ of conductor $\gf=\gp^t\overline{\gp}^s$ and infinity type $(q,r)$ satisfying $0\leqslant q,r \leqslant k$, we have
\begin{equation}{\label{interpolation no p-estab}}
L_p(\theta_{\varphi/K}^p,\psi_{p-\mathrm{fin}})=E_p(\theta_{\varphi/K}^p)\left[ \frac{D_Kw_K\tau(\psi)}{(-1)^{k+q+r}2 \varphi(\overline{\gp})^{t+s}\Omega_{\theta_{\varphi/K}}} \right] \Lambda(\theta_{\varphi/K},\psi),
\end{equation}
where 
\begin{equation*}
E_p(\theta_{\varphi/K}^p)=\prod_{\gq|p} \left(1-\frac{\varphi(\gp)\psi(\gq)}{N(\gp)}\right) \left(1-\frac{1}{\varphi(\overline{\gp})\psi(\gq)}\right).
\end{equation*}

\begin{theorem}{\label{T9.4}}
Let $\varphi$ be a Hecke character of $K$ with conductor $\gM$ coprime to $p$ and infinity type $(-k-1,0)$ for $k\geqslant0$, and denote by $\theta_{\varphi}^p$ the ordinary $p$-stabilization of the CM form $\theta_{\varphi}$ induced by $\varphi$. Let $\theta_{\varphi/K}^p$ be the base change to $K$ of $\theta_{\varphi}^p$ and $L_p(\theta_{\varphi/K}^p,-)$ its $p$-adic $L$-function, then for all  $\kappa\in\mathfrak{X}(\mathrm{Cl}_K(p^\infty))$ we have 
\begin{equation*}
L_p(\theta_{\varphi/K}^p,\kappa)=\frac{L_p(\varphi_{p-\rm{fin}}^c\kappa \chi_p)L_p(\varphi_{p-\rm{fin}}^c\kappa^c \lambda_p\chi_p)}{\Omega_p^{2k+2}},
\end{equation*}
where $\lambda_p,\chi_p$ are the $p$-adic avatars of the character $\lambda_K$ and the adelic norm $|\cdot|_{\A_K}$ respectively, and $\Omega_p$ is the $p$-adic period in Theorem \ref{Katztheorem}.
\end{theorem}

\begin{remark}{\label{acomodandoKatz}}
Note that $L_p(\theta_{\varphi/K}^p,-)$ is a function on $\mathfrak{X}(\mathrm{Cl}_K(p^\infty))$ while the right-hand side in the above theorem is 
a function on $\mathfrak{X}(\mathrm{Cl}_K(\overline{\gM} p^\infty)$. To relate them, we see the latter as a function on $\mathfrak{X}(\mathrm{Cl}_K(p^\infty))$ via the map $\mathrm{Cl}_K(\overline{\gM} p^\infty)\rightarrow\mathrm{Cl}_K(p^\infty)$.
\end{remark}
\begin{proof}
Since $L_p(\theta_{\varphi/K}^p,-)$ and $L_p(-)$ are measures, to obtain the equality in the theorem, it suffices to prove it on $p$-adic characters $\psi_{p-\text{fin}}$ coming from finite order Hecke characters $\psi$ of conductor $\gf=\gp^t\overline{\gp}^s$. 

For such characters, from (\ref{interpolation no p-estab}) and (\ref{gammafactors}) we have
\begin{equation*}
L_p(\theta_{\varphi/K}^p,\psi_{p-\mathrm{fin}})=E_p(\theta_{\varphi/K}^p)\left[ \frac{D_Kw_K\tau(\psi)}{(-1)^{k+1}2 \varphi(\overline{\gp})^{t+s}2^2} \right] \frac{L(\theta_{\varphi/K},\psi,1)}{\pi^2\Omega_{\theta_{\varphi/K}}}.
\end{equation*}
Recall from Lemma \ref{Prop 9.1} that
\begin{equation}\label{e: L-funct}
L(\theta_{\varphi/K},\psi,1)=L(\varphi^c\psi|\cdot|_{\mathbb{A}_K},0)L(\varphi^c\psi^c\lambda_K|\cdot|_{\mathbb{A}_K},0),    
\end{equation}
and denote the Hecke characters above by $\eta=\varphi^c\psi|\cdot|_{\mathbb{A}_K}$ and by $\eta'=\varphi^c\psi^c\lambda_K|\cdot|_{\mathbb{A}_K}$. By Theorem \ref{Katztheorem} we obtain the following interpolations
\begin{equation*}
\frac{L_p(\eta_{p-\text{fin}})}{\Omega_p^{k+1}}=\frac{\sqrt{D_K}^{-k}\tau_{\gp}(\eta)}{2(-1)^{k-1}w_K^{-1}2^{-k}N(\gp^{t})} \left(1-\frac{\varphi(\gp)\psi(\overline{\gp})}{N(\gp)}\right) \left(1-\frac{1}{\varphi(\overline{\gp})\psi(\gp)}\right)\frac{L(\eta,0)}{\pi^{-k}\Omega_\infty^{k+1}}
\end{equation*}
and 
\begin{equation*}
\frac{L_p(\eta'_{p-\text{fin}})}{\Omega_p^{k+1}}=\frac{\sqrt{D_K}^{-k}\tau_{\gp}(\eta')}{2(-1)^{k-1}w_K^{-1}2^{-k}N(\gp^{s})} \left(1-\frac{\varphi(\gp)\psi(\gp)}{N(\gp)}\right) \left(1-\frac{1}{\varphi(\overline{\gp})\psi(\overline{\gp})}\right)\frac{L(\eta',0)}{\pi^{-k}\Omega_\infty^{k+1}}.
\end{equation*}
The local Gauss sums
\begin{equation*}
\tau_p(\eta)=\tau_{\gp}(\varphi^c\psi|\cdot|_{\mathbb{A}_K})=(\varphi^c\psi)(\varpi_{\gp}^{-t})N(\gp^{t})\sum_{u\in(\mathcal{O}_{\gp}/{\gp}^t)^\times}\psi_{\gp}(u)\textbf{e}_K(ud_\gp^{-1}\varpi_{\gp}^{-t}),
\end{equation*}
\begin{equation*}
\tau_p(\eta')=\tau_{\gp}(\varphi^c\psi^c\lambda_K|\cdot|_{\mathbb{A}_K})=(\varphi^c\psi^c)(\varpi_{\gp}^{-s})N(\gp^{s})\sum_{v\in(\mathcal{O}_{\overline{\gp}}/{\overline{\gp}}^s)^\times}\psi_{\overline{\gp}}(v)\textbf{e}_K(vd_{\overline{\gp}}^{-1}\varpi_{\overline{\gp}}^{-s});
\end{equation*} 
are related with $\tau(\psi)$ (see for example ii) in \cite[Prop 2.14]{narki}) by
\begin{align*}
\tau_p(\eta)\tau_p(\eta')&=N(\gp^{t+s})\varphi^c(\varpi_{\gp}^{-t-s})\psi^{-1}(x_\gf)\sum_{b\in(\mathcal{O}_K/\gf)^\times}\psi_\gf(b)\textbf{e}_K(b d^{-1}x_\gf^{-1})\\
&=N(\gp^{t+s})\varphi(\overline{\gp})^{-t-s}\tau(\psi),
\end{align*}
where $x_\gf$ is the idele corresponding to $\gf$ satisfying $(x_\gf)_\gq=1$ for $\gq$ coprime to $(p)$, $(x_\gf)_\gp=\varpi_\gp^{t}$ and $(x_\gf)_{\overline{\gp}}=\varpi_{\overline{\gp}}^{s}$.

Rearranging and using (\ref{e: L-funct}) we have

\begin{equation*}
\frac{L_p(\eta_{p-\text{fin}})L_p(\eta'_{p-\text{fin}})}{\Omega_p^{2k+2}}= E_p(\theta_{\varphi/K}^p)\left[ \frac{D_Kw_K\tau(\psi)}{(-1)^{k+1}2 \varphi(\overline{\gp})^{t+s}2^2} \right] \frac{L(\theta_{\varphi/K},\psi,1)}{\Omega'_{\theta_{\varphi/K}}},
\end{equation*}
where $\Omega'_{\theta_{\varphi/K}}=\frac{2}{w_K}\left(\frac{\sqrt{D_K}\Omega_\infty}{2\pi i}\right)^{2k+2}$.

Finally, the result follows by normalizing the period $\Omega_{\theta_{\varphi/K}}$ in Definition \ref{periodberger} to be $\Omega'_{\theta_{\varphi/K}}$ and using the same argument in the proof of Proposition \ref{equality of measures}.
\end{proof}

\bibliographystyle{amsalpha}
\bibliography{main}
\end{document}